%% file: main_Maria.tex
\documentclass[11pt]{amsart}
\usepackage[utf8]{inputenc}
\usepackage{thm-restate}

\usepackage{enumerate,amsthm,todonotes, amssymb,xcolor,comment,subcaption}
\usepackage[footskip=1cm, headheight = 16pt, top=3.7cm, bottom=3cm,  right=2.5cm,  left=2.5cm ]{geometry}
\usepackage[colorlinks,linkcolor=blue,citecolor=red]{hyperref}
\usepackage{geometry}
\newtheorem{question}{Question}[section]
\newtheorem{lemma}[question]{Lemma}
\newtheorem{theorem}[question]{Theorem}
\newtheorem{conjecture}[question]{Conjecture}
\newtheorem{corollary}[question]{Corollary}

\newcounter{tbox}
\newcommand{\sta}[1]{\vspace*{0.3cm}\refstepcounter{tbox}\noindent{ \parbox{\textwidth}{(\thetbox) \emph{#1}}}\vspace*{0.3cm}}

\usepackage[T1]{fontenc}
\usepackage[leqno]{amsmath}
\usepackage{todonotes}
\usepackage[shortlabels]{enumitem}
\usepackage{latexsym}
\usepackage{amsfonts}
\usepackage{tikz}
\usepackage{float}
\usepackage{lmodern}
\usepackage[foot]{amsaddr}
\usepackage{latexsym}
\usepackage{tikzit}
\input{graph.tikzstyles}
\input{graph.tikzdefs}
\usepackage{algorithm}%
\usepackage{algorithmicx}%
\usepackage{algpseudocode}%
\usepackage{comment}
\usepackage{physics}
\usepackage{graphicx}
\usepackage{mathtools}
\usepackage{amssymb}
\usepackage{bm}
\usepackage{cleveref}
\usepackage{amsthm}
\usepackage{empheq}
\usepackage{amsmath}

\makeatletter
\newcommand{\leqnomode}{\tagsleft@true}
\newcommand{\reqnomode}{\tagsleft@false}
\makeatother

\def\dd{\hbox{-}}

\usepackage{bm}
\usepackage{bbm}
\usepackage{thmtools}

\newcommand{\eps}{\varepsilon}

\newcommand{\set}[1]{\left\{#1\right\}}
\newcommand{\mt}{\emptyset}

\newcommand{\reel}{{\mathbb R}}

\newcommand{\nat}{{\mathbb N}}

\newcommand{\NN}{\mathbb{N}}

\newcommand{\sm}{\setminus}

\newcommand{\tw}{\mathsf{tw}}

\theoremstyle{definition}

\tikzset{every picture/.style={line width=0.75pt}} 

\title{Tree independence number V. Walls and claws}

\author{Maria Chudnovsky$^{\dagger}$}
\author{Julien Codsi $^{\mathsection}$}
\author{Daniel Lokshtanov$^{\ddagger}$}
\author{Martin Milani\v{c}$^{\star}$}
\author{Varun Sivashankar $^{\mathsection \parallel}$}
\thanks{$^{\dagger}$ Princeton University, Princeton, NJ, USA. Supported by NSF Grant
 DMS-2348219 and by AFOSR grant FA9550-22-1-0083.}
\thanks{$^{\mathsection}$ Princeton University, Princeton, NJ, USA. Supported by NSF-EPSRC Grant DMS-2120644 and by AFOSR grant FA9550-22-1-0083.}
\thanks{$^{\star}$ FAMNIT and IAM, University of Primorska, Slovenia. Supported in part by the Slovenian Research and Innovation Agency (I0-0035, research program P1-0285 and research projects J1-3001, J1-3002, J1-3003, J1-4008, J1-4084, and N1-0370) and by the research program CogniCom (0013103) at the University of Primorska.}
\thanks{$^{\ddagger}$ Department of Computer Science, University of California Santa Barbara, Santa Barbara, CA, USA}
\thanks{$^{\mathsection \parallel}$ Princeton University, Princeton, NJ, USA. Supported by AFOSR grant FA9550-22-1-0083.}

\date{}

\begin{document}

\begin{abstract}
Given a family $\mathcal{H}$ of graphs, we say that a graph $G$ is $\mathcal{H}$-free if no induced subgraph of $G$ is isomorphic to a member of $\mathcal{H}$.
Let $S_{t,t,t}$ be the graph obtained from  $K_{1,3}$ by subdividing each edge $t-1$ times, and let $W_{t\times t}$ be the $t$-by-$t$ hexagonal grid.
Let $\mathcal{L}_t$ be the family of all graphs $G$ such that $G$ is the line graph of some subdivision of $W_{t \times t}$.
We prove that for every positive integer  $t$ there exists $c(t)$ such that  every
$\mathcal{L}_t \cup \{S_{t,t,t},  K_{t,t}\}$-free $n$-vertex  graph  admits a tree decomposition in which
the maximum size of an independent set in each bag is at most $c(t)\log^4n$.
This is a variant of a conjecture of Dallard, Krnc, Kwon, Milani\v{c}, Munaro, \v{S}torgel, and Wiederrecht from 2024.
This implies that  the \textsc{Maximum Weight Independent Set} problem, as well as
 many  other natural algorithmic problems,
  that are  known to be
  \textsf{NP}-hard in general, can be solved in quasi-polynomial time if
  the input graph is $\mathcal{L}_t \cup \{S_{t,t,t},K_{t,t}\}$-free.
As part of our proof, we   show  that for every positive integer  $t$ there exists an integer $d$ such that  every
$\mathcal{L}_t \cup \{S_{t,t,t}\}$-free graph admits a balanced separator that is contained
in the neighborhood of at most $d$ vertices.
\end{abstract}

\maketitle

\section{Introduction} \label{sec:intro}
All graphs in this paper are finite and simple and all logarithms are base $2$.
Let $G = (V(G),E(G))$ be a graph. For a set $X \subseteq V(G)$ we denote by $G[X]$ the subgraph of $G$ induced by $X$, and by $G \setminus X$ the subgraph of $G$ induced by $V(G) \setminus X$. In this paper, we use induced subgraphs and their vertex sets interchangeably.
For graphs $G,H$ we say that $G$ {\em contains $H$} if $H$ is
isomorphic to $G[X]$ for some $X \subseteq V(G)$. In this case, we say that
$X$ {\em is an $H$ in $G$}.  We say that
$G$ is {\em $H$-free} if $G$ does not contain $H$.
For a family $\mathcal{H}$ of graphs, we say that $G$ is
{\em $\mathcal{H}$-free} if $G$ is $H$-free for every $H \in \mathcal{H}$.

Let $v \in V(G)$. The \emph{open neighborhood of $v$}, denoted by $N_G(v)$, is the set of all vertices in $V(G)$ adjacent to $v$. The \emph{closed neighborhood of $v$}, denoted by $N_G[v]$, is $N(v) \cup \{v\}$. Let $X \subseteq V(G)$. The \emph{open neighborhood of $X$}, denoted by $N_G(X)$, is the set of all vertices in $V(G) \setminus X$ with at least one neighbor in $X$. The \emph{closed neighborhood of $X$}, denoted by $N_G[X]$, is $N_G(X) \cup X$. When there is
no danger of confusion, we omit the subscript $G$.
Let $Y \subseteq V(G)$ be disjoint from $X$. We say $X$ is \textit{complete} to $Y$ if all edges with an end in $X$ and an end in $Y$ are present in $G$, and $X$ is \emph{anticomplete}
to $Y$ if there are no edges between $X$ and $Y$.

For a graph $G$, a \emph{tree decomposition} $(T, \chi)$ of $G$ consists of a tree $T$ and a map $\chi\colon V(T) \to 2^{V(G)}$ with the following properties:
\begin{enumerate}
\itemsep -.2em
    \item For every $v \in V(G)$, there exists $t \in V(T)$ such that $v \in \chi(t)$.

    \item For every $v_1v_2 \in E(G)$, there exists $t \in V(T)$ such that $v_1, v_2 \in \chi(t)$.

    \item For every $v \in V(G)$, the subgraph of $T$ induced by $\{t \in V(T) \mid v \in \chi(t)\}$ is connected.
\end{enumerate}

For each $t\in V(T)$, we refer to $\chi(t)$ as a \textit{bag of} $(T, \chi)$.  The \emph{width} of a tree decomposition $(T, \chi)$, denoted by $width(T, \chi)$, is $\max_{t \in V(T)} |\chi(t)|-1$. The \emph{treewidth} of $G$, denoted by $tw(G)$, is the minimum width of a tree decomposition of $G$.  Graphs of bounded treewidth  are well understood both  structurally
\cite{RS-GMXVI} and algorithmically \cite{Bodlaender1988DynamicTreewidth}.

A {\em stable (or independent) set} in a graph $G$ is a set of pairwise non-adjacent vertices of $G$. The {\em stability (or independence) number} $\alpha(G)$ of $G$ is the maximum size of a stable set in $G$.  The  \textsc{Maximum Weight Independent Set (MWIS)} problem is
the problem whose input is a graph $G$ with weights on its vertices, and
whose output is a stable set of maximum total weight in $G$. \textsc{MWIS}
is known to be \textsf{NP}-hard \cite{alphahard}, but it can be solved
efficiently (in polynomial time)  in graphs of bounded treewidth.
Motivated by this fact,
Dallard, Milani\v{c}, and \v{S}torgel \cite{dms2} defined
a related graph width parameter, specifically targeting the complexity of the
\textsc{MWIS} problem.
The {\em independence number} of a tree decomposition
$(T, \chi)$ of $G$ is $\max_{t \in V(T)} \alpha(G[\chi(t)])$. The {\em tree independence number} of $G$, denoted tree-$\alpha(G)$, is the minimum independence number of a tree decomposition of $G$.
Results of Yolov~\cite{Yolov18} (or of Dallard et al.~\cite{dfgkm,dms2})  imply that \textsc{MWIS} can be solved in polynomial time on graphs of bounded tree independence number.
These results further imply that  \textsc{MWIS} can be solved in quasi-polynomial time in graph classes with tree independence number
polylogarithmic in the number of vertices.
Lima et al.~\cite{lima2024tree} observed that the algorithm of Dallard et al~\cite{dms2} can be extended to a much more general class of problems. We refer the reader to
\cite{TI2} for a detailed discussion of the algorithmic applications of polylogarithmic
bounds on the tree independence number. Tree independence number also has connections to coarse geometry \cite{CoarseGeoKwon, Agelos}.

Graph classes admitting useful bounds on their tree independence number were studied further
in \cite{dms4} and \cite{dms3}, where \cite{dms4} focused on excluding complete bipartite graphs.
In particular, the following was conjectured in \cite{dms4}.
Let $\mathcal{S}$ be the set of forests every component of which has at most three leaves, and for a graph $G$ let $L(G)$ denote the line graph of $G$.
\begin{conjecture}
\label{K2tpath}
For every positive integer $t$ and for all $S,T \in \mathcal{S}$, there
exists $c=c(t,S,T)$ such that every $\set{K_{t,t}, S, L(T)}$-free graph
has tree independence number at most $c$.
\end{conjecture}
In this paper, we study a variant of this conjecture.
Our main result is the following.
Let $t$ be a positive integer.
We denote by $S_{t,t,t}$ the graph obtained from the complete bipartite graph $K_{1,3}$ by subdividing each edge $t
-1$ times (so each edge is replaced by a $t$-edge path,
and $K_{1,3}$ is $S_{1,1,1}$). We call the unique degree-three vertex of $S_{t,t,t,}$  {\em the center} of $S_{t,t,t,}$.
We denote by $W_{t\times t}$  the $t$-by-$t$ hexagonal grid (also known as the
$t \times t$-wall).
Let $\mathcal{L}_t$ be the family of graphs $G$ for which there exists a
subdivision $H$ of   $W_{t \times t}$ such that $G=L(H)$. Let $\mathcal{M}_t$ be the class of all $\mathcal{L}_t \cup \{S_{t,t,t}, K_{t,t}\}$-free graphs. We prove:

\begin{theorem}
\label{thm_claw_treealpha}
For every positive integer $t$ there exists $c=c(t)$ such that
every $n$-vertex graph in $\mathcal{M}_t$ with $n\ge 2$ has tree independence number at most $c \log^4 n$.
    \end{theorem}
For a graph $G$, a function  $w\colon V(G) \rightarrow [0,1]$ is a {\em weight function} if $\sum_{v \in V(G)} w(v) \leq 1$.
For  $S \subseteq V(G)$, we write $w(S) \coloneqq  \sum_{v \in S} w(v)$.
A weight function $w$ is a {\em normal weight function} on $G$ if $w(V(G))=1$. If $0<w(V(G))<1$,  we call the function $w'\colon V(G) \rightarrow [0,1]$ given by $w'(v) = \frac{w(v)}{\sum_{u \in V(G)} w(u)}$ the
{\em normalized weight function of $w$}.
Let $c \in [0, 1]$ and let $w$ be a weight function on $G$.
A set $X \subseteq V(G)$ is a {\em $(w,c)$-balanced separator} if $w(D) \leq  c$ for every component $D$ of $G \setminus X$.
The set $X$ is a {\em $w$-balanced separator} if $X$ is a $(w,\frac{1}{2})$-balanced separator.
Given two sets of vertices $X$ and $Y $ of $G$, we say that $X$ is a {\em core}
for $Y$ if $Y \subseteq N[X]$.
A graph $G$ is said to be {\em $k$-breakable} if for every weight function $w\colon V(G) \rightarrow [0,1]$, there exists a $w$-balanced separator with a core of size strictly less than $k$. When the weight function $w$ is clear from the context, we might omit it from the notation.
Let $\mathcal{M}_t^*$ be the class of all $\mathcal{L}_t \cup \{S_{t,t,t}\}$-free graphs. As part of the proof of our main result, we show the following.

\begin{restatable}{theorem}{domsep}
  \label{thm:domsep}
    For every positive integer $t$, there is an integer $d=d(t)$ such that every graph $G \in \mathcal{M}_t^*$ is $d$-breakable.
\end{restatable}

This result is of independent interest. It is a significant step in the program of understanding induced subgraph (or induced minor) obstructions to small tree independence number (here by ``small'' we mean polylogarithmic in the size of the graph).
It  provides support for the following conjecture that was posed in~\cite{gartland2023quasi} and seems to be gaining popularity in the community:

\begin{conjecture}
    \label{conj:domsep}
    For every positive integer $t$, there is an integer $d=d(t)$ such that every $\mathcal{L}_t$-free graph $G$ with no induced subgraph isomorphic to a subdivision of the $t \times t$-wall is $d$-breakable.
\end{conjecture}
In turn, Conjecture~\ref{conj:domsep}, together with Theorem~\ref{few big independent neighborhoods in big independent set} and the methods of Section~\ref{sec:layeredsets}, are promising steps toward the following:
\begin{conjecture}
    \label{conj:smalltreealph}
    For every positive integer $t$, there is an integer $d=d(t)$ such that for every $n\ge 2$, every
    $n$-vertex graph with no induced minor isomorphic to $K_{t,t}$ or to $W_{t\times t}$ has tree independence number at most $\log^d n$.
\end{conjecture}

Also, together with Lemma~7.1 of \cite{TI2}, Theorem~\ref{thm:domsep} provides an alternative proof of Theorem~1.4 of \cite{dms4}, namely:

\begin{theorem}
For every positive integer $t$ and every pair of graphs $S,T\in \mathcal{S}$, there is an integer $d=d(t,S,T)$ such that the tree independence number of every $\{K_{1,t},S,L(T)\}$-free graph is at most $d$.
\end{theorem}

We remark that the majority of the proofs in this paper work in a slightly more general setup than excluding $S_{t,t,t}$, but we chose to present them in what we consider to be the most natural context.
A version  of Theorem~\ref{thm:domsep}  was recently proved independently in \cite{Marcin} by a somewhat different method.

\subsection{Proof outline and organization}

We start with Lemma~7.1 of \cite{TI2}:
\begin{lemma}\label{lemma:bs-to-treealpha}
  Let $G$ be a graph, let $c \in [\frac{1}{2}, 1)$, and let $d$ be a positive integer.
  If for every normal weight function $w$ on $G$, there is a
    $(w,c)$-balanced separator $X_w$ with $\alpha(X_w) \leq d$, then
        the tree independence number of $G$ is at most $\frac{3-c}{1-c}d$.
\end{lemma}

In view of Lemma~\ref{lemma:bs-to-treealpha}, in order to prove Theorem~\ref{thm_claw_treealpha} it is enough to show:

\begin{restatable}{theorem}{claw}
  \label{thm claw}
For every positive integer $t$, there exists $c=c(t)$ such that
for every $n\geq 2$, every $n$-vertex graph $G$ in $\mathcal{M}_t$,
and every normal weight function $w$ on $G$, there is a
$(w,\frac{1}{2})$-balanced separator $X_w$ in $G$ with
$\alpha(X_w) \leq c \log^4 n$.
\end{restatable}

We now outline the proof of Theorem \ref{thm claw}.

Our first goal is to prove Theorem~\ref{thm:domsep}.
An important tool in that proof is ``extended strip decompositions'' from
\cite{Threeinatree}. They are introduced in Section~\ref{sec:strips}, and in Section~\ref{sec:Hfacts} we prove several results about the behavior of extended strip decompositions in $\mathcal{L}_t$-free graphs.

The actual proof of Theorem~\ref{thm:domsep} is presented in Section~\ref{sec:domsep}; it proceeds as follows.
Let $G \in \mathcal{M}_t^*$. We may assume that $G$ is connected.
From now on we fix a weight function $w$ and assume that $G$ does not have a $w$-balanced separator with a small core.
By using the normalized weight function of $w$, we may assume that $w$ is normal. By Lemma~5.3 of \cite{QPTAS} there is an induced path $P=p_1 \dd \dots \dd p_k$ in $G$ such that $N[P]$ is a $w$-balanced separator in $G$. Choosing $P$ with $k$
minimum, we may assume that there is a component $B$ of $G \setminus N[P \setminus \{p_k\}]$
with $w(B) > \frac{1}{2}$. We now analyze the structure of the set $N=N(B) \subseteq N(P \setminus \{p_k\})$.
We say that $v \in N$ is a {\em hat} if $v$ has
exactly two neighbors in $P$, and they are adjacent.
First, we show that the set of all vertices in $N$ that are not hats has a small core. Now we focus on one hat $h$, and use it
to show that $G$ (with a subset with a small core deleted)  admits an extended strip decomposition. This allows us to produce a separator $S(h)$ with a small core that is not yet balanced but exhibits several useful properties. More explicitly, the component of
$G \setminus S(h)$ with maximum $w$-weight  only meets $P$ on one side of $h$. So $h$ either ``points left'' or ``points right''.
Then we show that the hat with the earliest neighbors in $P$ points right, and the hat with the latest neighbors in $P$ points left. Now we focus on two
consecutive hats $h,h'$ where the change first occurs, and conclude
that $S(h) \cup S(h')$ is a $w$-balanced separator in $G$.
 This completes the proof of Theorem~\ref{thm:domsep}.

Let us now continue with steps toward the proof of Theorem~\ref{thm claw},  and so assume that $G \in \mathcal{M}_t$.
Our next goal is to show that we can choose sets $Y_1, \dots, Y_{\lceil \log n \rceil}$ such that for every $j$,  $|Y_j| \leq d$,  $Y_j$ is a core of a $w$-balanced separator for $G$, and, for an appropriately chosen integer $D$, no vertex of $G$ belongs to more than $\frac{\log n}{D}$ of the sets $N[Y_j]$.
To do so, we continue selecting sets $Y_j$ as above, keeping track of the so-called ``layers'' $L_j^i$, which are sets of vertices that belong to at least $i$ out of the $j$ separators chosen so far.
We maintain the property that $\alpha(L_{j}^i)$ is bounded from above by a value that decreases geometrically with $i$ and increases geometrically with $j$, but at a much slower rate (see~\eqref{claim: small Level boosted} for details).
The main result of Section~\ref{sec:alpha_of_neighborhood} ensures that we are able to maintain this property by deleting a set of small stability number.
As a consequence, the sets $L_{\lceil \log n \rceil}^i$ are empty for large enough $i$, and that is what we needed to achieve.
We call this technique ``the layered set'' argument. This is done in Section~\ref{sec:layeredsets} (in fact, the result there is more general, to accommodate the proofs in Section~\ref{sec: better seps}).

Next, in Subsection~\ref{sec: improving separator} we
strengthen the conclusion of Theorem~\ref{thm:domsep} and establish the existence of a more refined type of separator, that we call
a {\em boosted separator}.
Let $(S,C)$ be a pair of subsets of $V(G)$ and let $B$ be a component of $G\setminus C$ with maximum weight.
For $\eps\in (0,\frac{1}{2}]$, the pair $(S,C)$ is said to be a {\em $(w,\eps)$-boosted separator} if $w(B)\leq 1/2$ or if $S$ is an $\eps$-balanced separator of $B$.
We call $C$ the {\em boosting set} of $(S,C)$.
A set $X\subseteq V(G)$ is said to be a {\em core} of $(S,C)$ if $S\subseteq N_G[X]$.   The existence of boosted separators with small cores,  where,
in addition, the boosting set  has a small stability number, is established in
Theorem~\ref{thm : boosting} by an application of the layered set argument
 (to a carefully chosen weight function different from $w$).

Now, another application of the layered set argument allows us to construct   sets $Y_1, \dots, Y_{\lceil \log n \rceil} $ and a set $C$
such that  $|Y_i| \leq d$ (where $d$ is a fixed integer), $C$ has small stability number, each  $(N[Y_i], C)$ is a $(w,\eps)$-boosted
separator, and   no vertex of $G$ belongs to more than $\frac{\log n}{D}$ of the sets $N[Y_i]$ (for appropriately chosen $\eps$ and $D$).  In
Subsection~\ref{sec: disjoint} we use this collection of sets to select
 a subcollection
 $Y_1, \dots, Y_{8t^2d}$  such that
no vertex of $G$ belongs to more than $t$ of the sets $N[Y_i]$. This is done in Theorem~\ref{thm:disjoint separators}.

We are now ready to put everything together. This is done in
 Section~\ref{sec:claw}, where  we use the results described above to complete the proof of Theorem~\ref{thm claw}.
By a first moment argument, we produce a large set $X$ of vertices such that
for at least $4t^2d$ of the sets
$Y_1,  \dots, Y_{8t^2d}$ above, no two vertices of $X$  belong to the same component of $G \setminus (C \cup N[Y_i])$; say these are sets $Y_1, \dots, Y_{4t^2d}$. Then we use a result of \cite{tw15} to describe the structure of a minimal
connected subgraph  $H$ of $G$ containing $X$. We deduce that $H$ contains a large set $\mathcal{P}$ of pairwise disjoint paths,
each of which meets at least $t$ of the sets $N[Y_1], \dots N[Y_{4t^2d}]$ above.
We consider minimal
subpaths of elements of $\mathcal{P}$ with this property.
Using the fact that no vertex of $G$ belong to $t$ of the sets $N[Y_i]$, together with several Ramsey-type arguments, we can construct paths $P_1, \dots, P_{3d}$ of length at least $t$, that are subpaths of distinct elements of $\mathcal{P}$, such that the  first vertex of each $P_j$ belongs to $N[Y_1]$ (say),
and the rest of $P_j$ is disjoint from $N[Y_1]$ (and therefore anticomplete to
$Y_1$). Since $|Y_1| \leq d$, it follows
that there exist $y \in Y_1$ and three paths $P,Q,R \in \{P_1, \dots, P_{3d}\}$ such that
$P \cup Q \cup R \cup \{v\}$ is an $S_{t,t,t}$ in $G$, a contradiction. This completes the proof
of \ref{thm claw}.

\section{Constricted sets and extended strip decompositions} \label{sec:strips}

An important tool in the proof of Theorem~\ref{thm:domsep} is ``extended strip decompositions'' of \cite{Threeinatree}. We explain this now.
A set $C \subseteq G$ is a {\em hole in $G$} if $G[C]$ is  cycle of length at least four.
Similarly, a set $P\subseteq G$ is a {\em path in $G$} if  $G[P]$ is a path. Let $P=\{p_1, \dots, p_k\}$
be a path in $G$ where $p_ip_j \in E(G)$ if and only if $|j-i|=1$. We say that $p_1$ and $p_k$ are
the {\em ends} of $P$. The {\em interior} of $P$, denoted by
$P^*$, is the set $P \setminus \{p_1,p_k\}$.
For $i,j \in \{1, \dots. k\}$ we denote by $p_i \dd P \dd p_j$ the subpath of
$P$ with ends $p_i,p_j$.
Let $G,H$ be graphs, and let $Z \subseteq V(G)$. Let $W$ be the set of vertices of degree one in $H$. Let $T(H)$ be the set of all triangles of $H$.
Let $\eta$ be a map with domain
the union of $E(H)$, $V(H)$, $T(H)$,  and the set of all pairs $(e, v)$ where $e \in E(H)$, $v \in V (H)$,   and $e$ incident
with $v$, and range $2^{V(G)}$,  satisfying the following conditions:
\begin{itemize}
\item For every $v \in V(G)$ there exists
$x \in E(H)  \cup V(H) \cup T(H)$ such that $v \in \eta(x)$.
\item If $x,y \in E(H) \cup V(H) \cup T(H)$ and $x \neq y$, then
$\eta(x) \cap \eta(y)=\emptyset$.
\item For every $e \in E(H)$ and $v \in V(H)$ such that $e$ is incident with $v$, $\eta(e,v) \subseteq \eta(e)$.
\item Let  $e,f \in E(H)$ with  $e \neq f$,  and $x \in \eta(e)$ and $y \in \eta(f)$. Then $xy \in E(G)$ if and only if
$e, f$ share an end-vertex $v$ in $H$, and $x \in \eta(e, v)$ and
$y \in  \eta(f, v)$.
\item If $v \in V (H)$, $x  \in \eta (v)$, $y \in V (G) \setminus \eta(v)$, and
$xy \in E(G)$, then $y \in \eta(e, v)$ for some
$e \in E(H)$ incident with $v$.
\item  If $D \in T(H)$,  $x \in \eta (D)$, $y \in V (G) \setminus \eta(D)$
and $xy \in E(G)$, then
$y \in \eta(e,u) \cap \eta(e,v)$  for some distinct $u, v \in D$, where
$e$ is the edge $uv$ of $H$.
\item  $|Z| = |W |$, and for each $z \in Z$ there is a vertex $w \in W$ such that $\eta(e, w) = \{z\}$, where $e$ is the
(unique) edge of $H$ incident with $w$.
\end{itemize}
Under these circumstances, we say that  $\eta$ is an {\em extended strip decomposition of
$(G,Z)$ with pattern $H$} (see Figure~\ref{fig:ex esd}).
Let  $e$ be an edge of $H$ with ends $u,v$. An {\em $e$-rung} in
$\eta$ is a path $p_1 \dd \dots \dd p_k$ (possibly $k=1)$ in $\eta(e)$,
with $p_1 \in \eta(e,v)$, $p_k \in \eta(e,u)$ and
$\{p_2, \dots, p_{k-1}\} \subseteq \eta(e) \setminus (\eta(e,v)  \cup \eta(e,u))$. We say that $\eta$ is  {\em faithful} if for every $e \in E(H)$,
there is an $e$-rung in $\eta$.

A set $A \subseteq V(G)$ is an {\em atom} of $\eta$ if one of the following holds:
\begin{itemize}
\item $A=\eta(x)$ for some $x \in V(H) \cup T(H)$.
\item $A=\eta(e) \setminus (\eta(e,u) \cup \eta(e,v))$ for some
edge $e$ of $H$ with ends $u,v$.
\end{itemize}
For an atom $A$ of $\eta$, the {\em boundary} $\delta(A)$ of $A$ is defined as follows:
\begin{itemize}
\item If $v \in V(H)$ and  $A=\eta(v)$,  then
$\delta(A)=\bigcup_{e \in E(H) \; \colon \; e \text{ is incident with } v} \eta(e,v)$.
\item If $A=\eta(D)$ , and $D \in T(H)$ with $D=v_1v_2v_3$, then
$\delta(A)=\bigcup_{i \neq j \in \{1,2,3\}}\eta(v_iv_j,v_i) \cap \eta(v_iv_j,v_j)$.
\item If $A=\eta(e) \setminus (\eta(e,u) \cup \eta(e,v))$ for some
edge $e$ of $H$ with ends $u,v$, then $\delta(A)=\eta(e,u) \cup \eta(e,v)$.
\end{itemize}
A set $Z \subseteq V (G)$ is {\em constricted} for every $T \subseteq G$ such that $T$ is a tree,  $|Z \cap V (T )| \leq  2$.

\begin{figure}[h]
    \centering
\includegraphics[width=0.8\linewidth]{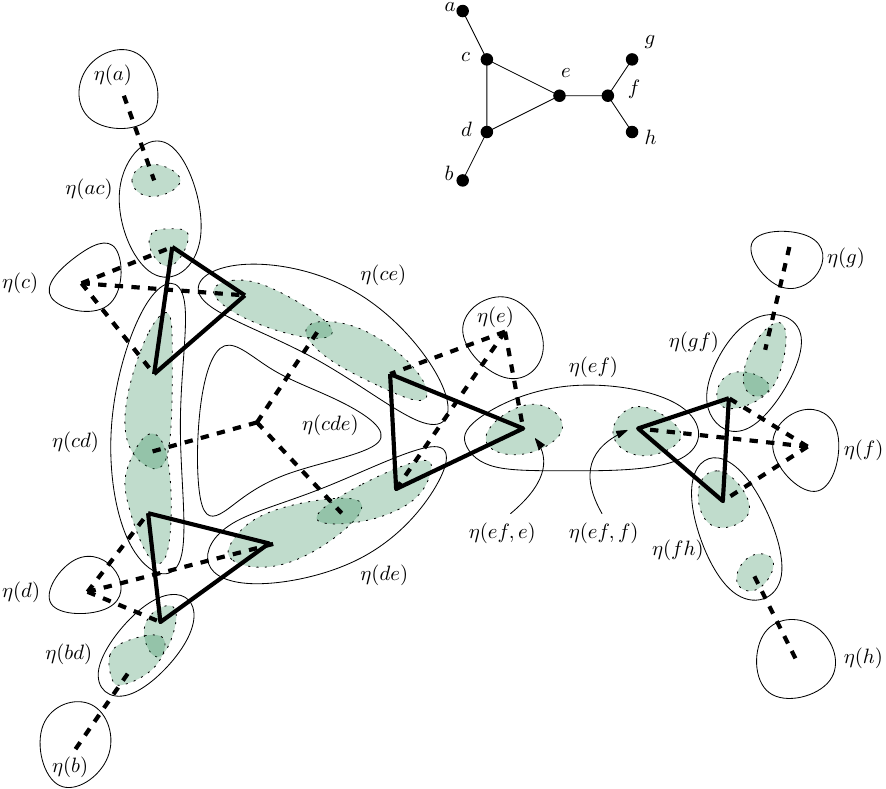}
    \caption{Example of an extended strip decomposition with its pattern (here dash lines represent potential edges).
    This figure was created by Pawe\l{} Rz\k{a}\.zewski and we use it with his permission.}
    \label{fig:ex esd}
\end{figure}

The main result of \cite{Threeinatree} is the following.
\begin{theorem}\label{stripdecomp}
Let $G$ be a connected graph and let $Z \subseteq  V (G)$ with $|Z| \geq 2$.
Then $Z$ is constricted if and only
if for some graph $H$, $(G, Z)$ admits a faithful extended strip decomposition with pattern $H$.
\end{theorem}

We also need the following, which is  an immediate corollary of Lemma 6.8 of \cite{QPTAS}:
\begin{lemma} \label{threepaths}
Let $G,H$ be graphs, $Z \subseteq V(G)$ with $|Z| \geq 3$,
and let $\eta$ be an  extended strip decomposition of $(G,Z)$
with pattern $H$. Let $Q_1,Q_2,Q_3$ be paths in $G$,
pairwise anticomplete to each other, and each with an end in $Z$.
Then for every atom $A$ of $\eta$, at least one of the sets
$N[A] \cap Q_1$, $N[A] \cap Q_2$ and $N[A] \cap Q_3$ is empty.
\end{lemma}

We finish this section with another lemma.
\begin{lemma}\label{atomboundary}
Let $G,H$ be graphs, $Z \subseteq V(G)$ with $|Z| \geq 2$,
and let $\eta$ be a faithful  extended strip decomposition of  $(G,Z)$
with pattern $H$. Let $A$ be an atom of $\eta$. Then $\delta(A)$ has a core of
size at most~$3$.
\end{lemma}

\begin{proof}
Suppose first that $A=\eta(v)$ for some $v \in V(H)$. If $v$ has degree one in $H$, then by the definition of an extended strip decomposition,
$|\eta(v,e)|=1$ (where $e$  is the unique edge incident with $v$), and so
$|\delta(A)|=1$. Thus we may assume that $v$ is incident with at least two edges, say $e$ and $f$, in $H$. Since $\eta$ is faithful, there exist
$x \in \eta(e,v)$ and $y \in \eta(f,v)$. But now
$\delta(A) \subseteq N[\{x,y\}]$, as required.

Next assume that $A=\eta(D)$  and $D \in T(H)$ with $D=v_1v_2v_3$.
Since $\eta$ is faithful, there exist $x \in \eta(v_1v_2,v_1)$
and $y \in \eta(v_1v_3, v_1)$ and $z \in \eta(v_1v_2,v_2)$. Now $\delta(A) \subseteq N[\{x,y,z\}]$
as required.

Thus we may assume that $A=\eta(e) \setminus (\eta(e,u) \cup \eta(e,v))$ for some
edge $e$ of $H$ with ends $u,v$. We may assume that the degree of $u$ in $H$ is at least~$2$, 
let $f \neq e$ be an edge of $H$ incident with $u$.
Since $\eta$ is faithful, there exists $x \in \eta(f,u)$, and
$\eta(e,u) \subseteq N(x)$.
If $v$ has degree one in $H$, then $|\eta(e,v)|=1$, and
$\delta(A) \subseteq N[\eta(e,v) \cup \{x\}]$, as required.
Thus we may assume that there is an edge $f' \neq e$ such that $v$ is incident with $f'$. Since $\eta$ is faithful, there exists $y \in \eta(f',v)$.
Now $\delta(A) \subseteq N(\{x,y\})$, and the conclusion of the theorem holds.
\end{proof}

\section{Extended strip decompositions in  graphs in $\mathcal{M}_t^*$} \label{sec:Hfacts}

In this section we prove several results about the behavior of extended strip decompositions in $\mathcal{L}_t$-free graphs, that we will use in the proof
of Theorem~\ref{thm:domsep}.

We need a result of \cite{chuzhoy}:
\begin{theorem}[\cite{chuzhoy}]\label{subgraphtw}
	There exist positive integers $c_1$ and $c_2$ such that for every positive integer $t$, every graph with no subgraph isomorphic to a subdivision of the $(t\times t)$-wall has treewidth at most $c_1t^9\log^{c_2}t$.
\end{theorem}

Our first goal is to prove the following:
\begin{theorem} \label{smalltwH}
Let $t$ be an integer, let $G$ be an $\mathcal{L}_t$-free graph, and let $Z \subseteq V(G)$. Let $\eta$ be a faithful extended strip decomposition of $G$ with pattern $H$. Then $\tw(H) \leq c_1t^9\log^{c_2}t$, where $c_1,c_2$ are as in Theorem~\ref{subgraphtw}.
\end{theorem}

\begin{proof}
By Theorem~\ref{subgraphtw} it is enough to show that no subgraph of $H$ is isomorphic to a subdivision of the $(t\times t)$-wall. Since $\eta$ is faithful,
for every $e \in E(H)$, we can choose
an $e$-rung $R_e$ in $\eta$. Let $G'=\bigcup_{e \in E(H)}R_e$.
Then there exists a graph $H'$, obtained from
$H$ by subdividing edges, such that $G'=L(H')$.
Now, let  $F$ be a subgraph of $H$  isomorphic to a subdivision of the
$(t\times t)$-wall. Let $G''=\bigcup_{e \in E(F)} R_e$. Then $G''$ is
 the line graph of a graph $F''$, where $F''$ is a graph obtained from $F$ by subdividing edges, contrary to the fact that $G$ is $\mathcal{L}_t$-free.
 \end{proof}

Before we embark on the proof of Theorem~\ref{thm:domsep}, we need one final lemma:

\begin{lemma} \label{lem:bigatom}
Let $t \geq 2$ be an integer. Let $G$ be an $\mathcal{L}_t$-free graph, and let $w$ be a weight function on $G$. Let $D$ be a component of $G$ with $w(D)>\frac{1}{2}$.
Let $Z \subseteq D$, and let $\eta$ be a faithful extended strip decomposition of $(D,Z)$ with pattern $H$. Assume that $w(A) \leq \frac{1}{2}$
for every atom $A$ of $\eta$. Let $c_1,c_2$ be as in Theorem~\ref{subgraphtw}.
Then there exists $Y \subseteq V(G)$
with $|Y|   \leq 3c_1t^9\log^{c_2}t$, such that $N[Y]$ is a $w$-balanced separator in $G$.
\end{lemma}

\begin{proof}
By working with the normalized weight function of $w$, we may assume that $w$ is normal.
Let $H'$ be obtained from $H$ as follows. Subdivide every edge $e$ of $H$ once;
call the new vertex $v_e$.
For every $v \in V(H)$, add a new vertex $v_v$ adjacent to $v$
and with no other neighbors. For every triangle $T=uvw$ of $H$,
add a vertex $v_T$ adjacent to $u,v,w$ with no other neighbors (see Figure~\ref{fig:modified pattern}).
Observe that vertices of $V(H') \setminus V(H)$ correspond to atoms of
$\eta$.
For every component $D'$ of $G \setminus D$, add a new isolated vertex
$v_{D'}$.
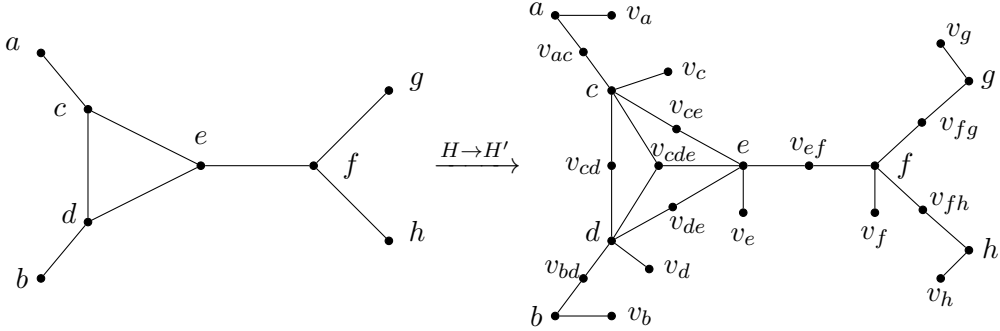
\begin{figure}[ht]
    \centering
    $\input{./figures/pattern.tikz} \xrightarrow{H\rightarrow H'} \input{./figures/pattern_prime.tikz}$
    \caption{Example of $H'$ given an $H$.}
    \label{fig:modified pattern}
\end{figure}

Now define a weight function $w'$ on $H'$.
For every $v \in V(H)$, let $w'(v)= w(\delta(\eta(v)))$, that is,
$w'(v)=
\sum_{e \in E(H) \text{ incident with }v}w(\eta(e,v))$, and let $w'(v_v)=w(\eta(v))$.
For every $e \in E(H)$ with ends $u,v$,
let $w'(v_e)=w(\eta(e)  \setminus (\eta(e,v) \cup \eta(e,u)))$.
For every triangle $T$ of $H$, let $w'(v_T)=w(\eta(T))$.
For every component $D'$ of $G \setminus D$, let $w'(v_{D'})=w(D')$.
Now $w'$ is a normal function on $H'$.

By Theorem~\ref{smalltwH} $\tw(H) \leq c_1t^9\log^{c_2}t$.
Since $H'$ is obtained from $H$ by subdividing edges and adding vertices
whose neighborhood is a clique of size at most three, it is easy to see that $\tw(H') \leq \max \{\tw(H)$, 3\}, and so $\tw(H') \leq c_1t^9\log^{c_2}t$.
By a result from~\cite{zbMATH06707229}, it follows that there exists a $w'$-balanced separator  $X' \subseteq V(H')$ with $|X'| \leq c_1t^9\log^{c_2}t$.
We may assume that for every component $D'$ of $G \setminus D$,
$v_{D'} \not \in X'$.
Next, we use $X'$ to obtain a $w$-balanced separator $X$ in $G$.
First, for every $v \in X' \cap V(H)$, add to $X$ the set
$\bigcup_{e \in E(H) \text{ incident with }v}\eta(e,v)$.
Second, for every $e \in E(H)$ with ends $u,v$  such that $v_e \in X'$,
add to $X$ the boundary of the atom $\eta(e) \setminus (\eta(e,u) \cup \eta(e,v))$.
Third, for every $v \in V(H)$ such that $v_v \in X'$,
add to $X$ the boundary of the atom $\eta(v)$.
Finally, for every triangle $T=uvw$ of $H$  with $v_T \in X'$,
add to $X$ the boundary of the atom $\eta(T)$.
This completes the construction of $X$.

Since by Lemma~\ref{atomboundary} for every atom $A$ of $\eta$,
$\delta(A)$ has a core of size at most
three, and since for every $v \in V(H)$,  the set $\bigcup_{e \in E(H) \text{ incident with }v}\eta(e,v)$ has a core of size at most two, it follows that
there exists $Y \subseteq V(H)$ with $|Y| \leq 3 c_1t^9\log^{c_2}t$
such that $X \subseteq N[Y]$.

It remains to show that $X$ is a $w$-balanced separator in $G$.
Suppose not, and let $C$ be a component of $G \setminus X$ with $w(C)>\frac{1}{2}$.
Let $U \subseteq V(H) \cup E(H) \cup T(H)$ be such that $u \in U$
if and only if one of the following holds:
\begin{itemize}
\item $\eta(u) \cap C \neq \emptyset$, or
\item $u \in V(H)$ and there exists $e \in E(H)$ incident with $u$ such that
$\eta(e,u) \cap C \neq \emptyset$.
\end{itemize}
Let $f:V(H) \cup E(H) \cup T(H) \rightarrow V(H') \setminus V(H)$,
where $f(x)=v_x$, and let $U'=(U \cap V(H)) \cup f(U)$.
Then $U' \subseteq V(H')$.

\sta{$U' \cap X' = \emptyset$. \label{U'X'}}

Suppose that there exists $u' \in U' \cap X'$.
Define $u\in U$ as follows.
If $u'\in U\cap V(H)$, let $u = u'$.
If $u'\in f(U)$, let $u\in U$ be such that $u' = v_{u}$.
Assume that $\eta(u)\cap C=\emptyset$.
Since $u'\in U'$, it follows that $u  = u'$, and there exists $e \in E(H)$ incident with $u$ such that $\eta(e,u) \cap C \neq \emptyset$.
However, $\eta(e,u) \subseteq X$ since $u'\in X'$, contrary to the fact that $C \cap X = \emptyset$. This proves that $\eta(u) \cap C \neq \emptyset$ for every $u' \in U'\cap X'$.
Assume first that $u'=v_e$ for some edge $e \in E(H)$ with ends $x,y$.
Then $u = e$ and $\eta(e,x) \cup \eta(e,y) \subseteq X$, and so, since $C$ is connected and $C\cap \eta(e)\neq\emptyset$, it follows that  $C \subseteq \eta(e) \setminus (\eta(e,x) \cup \eta(e,y))$.
But then $C$ is a subset of an atom  of $\eta$, and so $w(C) \leq \frac{1}{2}$, a contradiction.
It follows that $u'=v_x$ for some $x \in V(H) \cup T(H)$.
Then $\delta(\eta(x)) \subseteq X'$, and so, since $C$ is connected, it follows that  $C \subseteq \eta(x)$. Again
$C$ is a subset of an atom  of $\eta$, and so $w(C) \leq \frac{1}{2}$, a contradiction.
This proves~\eqref{U'X'}.
\\
\\
Since $C$ is connected, we deduce that $U'$ is connected, and therefore, by
\eqref{U'X'},  $U'$ is contained in a component of $H' \setminus X'$.
But $w(C) \leq \sum_{u \in U'} w'(u)$, contrary to the fact that $X'$
is a $w$-balanced separator of $H'$.
\end{proof}

\section{Bounded core separators in graphs in $\mathcal{M}_t^*$}

\label{sec:domsep}
We are now ready to complete the  first step in the proof of Theorem~\ref{thm claw}, that is, Theorem \ref{thm:domsep}, which we restate.

\domsep*

\begin{proof}
We may assume that $t \geq 2$.
Let $G \in \mathcal{M}^*_t$ and let $w$ be a weight function on $G$.
By working with the normalized function of $w$, we may assume that $w$ is normal.
Let $c_1,c_2$ be as in Theorem~\ref{subgraphtw}.
Let $d=3c_1t^9\log^{c_2}t+22t$.
We will show that there is a set $Y \subseteq G$ with $|Y|<d$ such that
$N[Y]$ is a $(w,\frac{1}{2})$-balanced separator in $G$. Suppose no such $Y$
exists.

By the proof of Lemma~5.3 of \cite{QPTAS} there is a path $P$ in $G$ such that $N[P]$ is a $w$-balanced separator in $G$. Let $P=p_1 \dd \dots \dd p_k$, and assume that $P$ was chosen with $k$ minimum. It follows that there exists a component $B$ of
$G \setminus N[P \setminus \{p_k\}]$ such that $w(B) > \frac{1}{2}$. Let
$N=N(B)$. Then $N \subseteq N(P \setminus \{p_k\})$.
First, we show

\sta{There is no $Y \subseteq G$ with  $|Y|<d$ such that $N \cup N[p_k]  \subseteq N[Y]$. \label{goodY}}

Suppose such $Y$ exists. We will show that $N[Y]$ is a $w$-balanced separator in $G$. We may assume that there is a component  $D$ of  $G \setminus N[Y]$ with
$w(D)>\frac{1}{2}$. Since $w(B)>\frac{1}{2}$, we deduce that
$D \cap B \neq \emptyset$.  Since $N \subseteq N[Y]$, it follows that
$D \subseteq B$, and so $D \cap N[P] \subseteq N[p_k]$. Since
$N[p_k] \subseteq N[Y]$, we deduce that $D$ is contained in a component of $G \setminus N[P]$, and therefore $w(D) < \frac{1}{2}$,
a contradiction. This proves that  $N[Y]$ is a $w$-balanced separator in $G$,
contrary to our assumption, and \eqref{goodY} follows.
\\
\\
Let $a,b\in \mathbb{Z}_{\ge 0}$ such that $k=2at+b$ and $b<2t$. For $i \in \{1, \dots, a-1\}$ let $P_i=p_{2(i-1)t+1} \dd \dots \dd p_{2it}$, and let $P_a=p_{2(a-1)t+1} \dd \dots \dd p_k$.
Let $Y_1=P_1 \cup   P_a$. Then $|Y_1| < 6t$.
Let $N_1=N \setminus N[Y_1]$. We deduce from  \eqref{goodY}:

\sta{There is no $Y \subseteq G$ with  $|Y| \leq d-6t$ such that $N_1  \subseteq N[Y]$. \label{N1}}

In the next several arguments, we will treat the two ends of $P$ symmetrically
since we will only use the property that $N(B) \subseteq N(P \setminus \{p_k\})$. We will state explicitly when additional properties of $P$ come into play, and stop using this symmetry. We call $v \in N_1$ a
{\em hat} if $v$ has exactly two neighbors in $P$, and these neighbors are consecutive in $P$.
Let $H_1 \subseteq N_1$ be the set of all hats. Our next goal is to reduce the problem to the case when $N_1=H_1$.
\\
\\
\sta{For every $v \in N_1 \setminus H_1$, $\alpha (N(v) \cap P) \geq 3$. \label{smallcases}}

Suppose that there is $v \in N_1 \setminus H_1$ with $\alpha(N(v) \cap P) \leq 2$.
Let $r$ be minimum and $s$ be maximum such that $v$ is adjacent to $p_r,p_s$.
Since $v \in N_1$, it follows that $r>2t$ and $s \leq k-2t$.
Since $\alpha(N(v) \cap P) \leq 2$, we deduce that
$N(v) \cap P \subseteq \set{p_r,p_{r+1},p_{s-1},p_s}$.
Let $R=p_{r-2t+1} \dd P \dd p_{r+2t+1}$  and let $S=p_{s-2t} \dd P \dd p_{s+2t}$.
Let $Z=R \cup S$. Then $|Z| \leq 2(4t+2) \leq  d- 6t$, and so there
exists $w \in N_1 \setminus N(Z)$.

Since $v,w$ are both in $N_1$, there is a path $Q$ from $v$ to $w$ with
$Q^* \subseteq B$. Let $i$ be minimum and $j$ be maximum such that $w$ is adjacent to $p_i,p_j$. Since $w \in N_1$, it follows that $i>2t$ and $j \leq k-2t$.

Suppose first that $r=s$. Since  $w$ is anticomplete to $R$, it follows that $p_i  \not \in R$. Now we get an $S_{t,t,t}$ with center $p_r$  two of whose paths are subpaths of   $R^*$ and the third is a subpath of $p_r \dd v \dd Q \dd w \dd p_i  \dd P \dd p_{i-t+3}$, a contradiction.

This proves that $r \neq s$, and since $v \not \in H_1$, it follows that
$s>r+1$. Now  we get  an $S_{t,t,t}$ with center $v$ whose   paths are
$v \dd p_r \dd R \dd p_{r-t+1}$, $v \dd p_s \dd S \dd p_{s+t-1}$ and
a subpath of $v \dd Q \dd w \dd p_i \dd P\dd p_{i-t+2}$, again a contradiction.
This proves \eqref{smallcases}.
\\
\\
\sta{There do not exist $1<i<j<\ell<a$ and $v \in N_1$ such that
$v$ has a neighbor in $P_i$ and a neighbor in $P_\ell$, and $v$
is anticomplete to $P_j$. \label{noskip}}
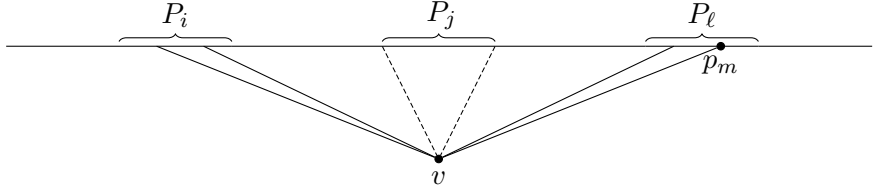
\begin{figure}[ht]
    \centering
    \input{./figures/d_separable_non_consecutive.tikz}
    \caption{Visualization for \eqref{noskip}.}
    \label{fig: no skip}
\end{figure}

Suppose such $i,j,\ell,v$ exist (see Figure~\ref{fig: no skip}).
Then $v \not \in H_1$. By \eqref{smallcases}
we may assume that $v$ has two non-adjacent neighbors in $p_1 \dd P \dd p_{2t(j-1)}$. It follows that there exist subpaths $Q,R$ with $|Q|=|R|=t$ of
$p_1 \dd P \dd p_{2t(j-1)+t}$ and anticomplete to each other such that
$v$ is adjacent to exactly one end of $Q$ and has no other neighbors in $Q$,
and $v$ is adjacent to exactly one end of $R$ and has no other neighbors in
$R$. Let $m$ be maximum such that $v$ is adjacent to $p_m$. Since
$v \in N_1$, it follows that $m \leq k-2t$. Now we get an $S_{t,t,t}$ with
center $v$ and path $v \dd Q$, $v \dd R$ and $v \dd p_m \dd P \dd p_{m+t-1}$,
a contradiction. This proves~\eqref{noskip}.
\\
\\
\sta{There exist $i,j \in \{2, \dots, a-1\}$ such that $N_1 \setminus H_1 \subseteq N(P_i) \cup N(P_j)$. \label{nonhats}}

Let $\mathcal{I} \subseteq \{2, \dots, a-1\}$ be such that $N_1 \setminus H_1 \subseteq \bigcup_{i \in \mathcal{I}}N(P_i)$ and with $|\mathcal{I}|$ minimum. We may assume that
$|\mathcal{I}| \geq 3$; let $i,j,\ell \in \mathcal{I}$ with $i<j<\ell$.
By the minimality of $|\mathcal{I}|$, there exist $v_i,v_\ell \in N_1 \setminus H_1$ such that
$v_i \subseteq N(P_i) \setminus N(P_j)$ and  $v_\ell \subseteq N(P_\ell) \setminus N(P_j)$. By \eqref{noskip}, we have that $v_i$ is anticomplete to
$\bigcup_{m \geq j}P_m$ and $v_\ell$ is anticomplete to $\bigcup_{m \leq j}P_m$.
Since both $v_i,v_\ell \in N_1$, there is a path $Q$ from $v_i$ to $v_\ell$ with
$Q^* \subseteq B$. Let $r$ be minimum and $s$ be maximum such that
$v_i$ is adjacent to $p_r,p_s$.
Then $s \leq 2t(j-1)$ and, by \eqref{smallcases}, $r+1<s$. Since $v_i \in N_1$, it follows that
$r > 2t$. Let $q$ be minimum such that $v_\ell$ is adjacent to $p_q$. Then
$q>2tj$.
 Now we get an $S_{t,t,t}$ with center $v_i$
and whose paths are $v_i \dd p_r \dd P \dd p_{r-t+1}$, $v_i \dd p_s \dd P \dd p_{s+t-1}$,
and a subpath of $v_i \dd Q \dd v_\ell  \dd p_q \dd P \dd p_{q-t+2}$, a contradiction.
This proves~\eqref{nonhats}.
\\
\\
By \eqref{nonhats}, there exist $i', j' \in \{2, \dots, a-2\}$ (possibly $i'=j'$) such that $N_1 \setminus H_1 \subseteq N(P_{i'}) \cup N(P_{j'})$.
Let $Y_2=P_{i'} \cup P_{j'}$. Then $|Y_1 \cup Y_2| < 10t$ and $N_1 \setminus H_1 \subseteq N(Y_1 \cup Y_2)$. Let $H_2=H_1 \setminus N(Y_1 \cup Y_2)$. By \eqref{N1},
$H_2 \neq \emptyset$.

From now on we  will use additional properties of $P$, so we no longer
use symmetry between its two ends.
Let $x$ be minimum such that there exists $h_0 \in H_2$ with $N(h_0) \cap P=\{p_x,p_{x+1}\}$.
Let $y$ be maximum  such that there exists $h_1 \in H_2$ with $N(h_1) \cap P=\{p_{y},p_{y+1}\}$; we refer to this later as the {\em  maximality of $h_1$}.
(See Figure~\ref{fig:h_0 h_1}.)

    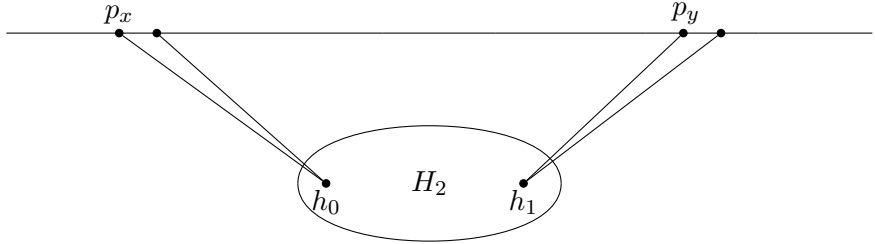
\begin{figure}[ht]
        \centering
    \input{./figures/d_separable_h_0_h_1.tikz}
    \caption{Visualization of $h_0$ and $h_1$.}
    \label{fig:h_0 h_1}
\end{figure}

For all $h \in H_2$, we define a subpath $P(h)$ of $P$ as follows: $P(h)=p_{i-2t-1} \dd P \dd p_{i+2t+2}$, where $N(h) \cap P=\{p_i,p_{i+1}\}$.
Let $H_3 = H_2\setminus N(P(h_0))$.

\sta{$H_3\neq \emptyset$. \label{H3}}

Suppose $H_3 = \emptyset$.
Let $Y = Y_1\cup Y_2\cup P(h_0)$.
Now $|Y| < 16t < d$ and $N\cup N[p_k]\subseteq N[Y]$, contrary to \eqref{goodY}.
This proves~\eqref{H3}.
\\
\\
Our next goal is to define, for every $h\in H_3$, a graph $G_h$, a triple $(z_1(h),z_2(h),z_3(h))$ of vertices of $G_h$, and an extended strip decomposition $\eta_h$ of $(G_h,\{z_1(h),z_2(h),z_3(h)\})$.
So let $h\in H_3$ and let $N(h) \cap P=\{p_i,p_{i+1}\}$.
Let $z_1(h)=p_{i-2t-1}$, $z_2(h)=p_{i+2t+2}$ and $z_3(h)=h$.
Write  $P_L(h)=p_1 \dd P \dd p_{i-2t-2}$ and
$P_R(h)=p_{i+2t+3} \dd P \dd p_k$.
We define
\[G_h' = (G \setminus N[Y_1 \cup Y_2 \cup P(h_0) \cup P(h)]) \cup (P_L(h)  \cup P_R(h) \cup (N(p_k) \cap B) \cup \{ z_1(h), z_2(h), z_3(h)\})\,.\]
For $i \in \{1,2,3\}$ write $z_i=z_i(h)$.
Then $z_1,z_2,z_3 \in G_h'$, and $B \subseteq G_h'$.
Let $G_h$ be the component of $G_h'$ containing $B$.
Then $z_3\in G_h$.
Since $p_k\in G_h'$ and $p_k$ has a neighbor in $B$, it follows that $p_k\in G_h$, and consequently $z_2\dd P \dd p_k\subseteq G_h$.
Since $h_0$ has a neighbor in $B$,
it follows that $h_0\in G_h$; and, since $h\in H_3$, we deduce that $p_1\dd P \dd z_1\subseteq G_h$. (See Figure~\ref{fig:Gh}.)

\begin{figure}[ht]
    \centering
    \scalebox{1.5}{\input{./figures/d_separable_gh.tikz} }
    \caption{Visualization of $G_h$ (here dashed lines represent non-edges).}
    \label{fig:Gh}
\end{figure}
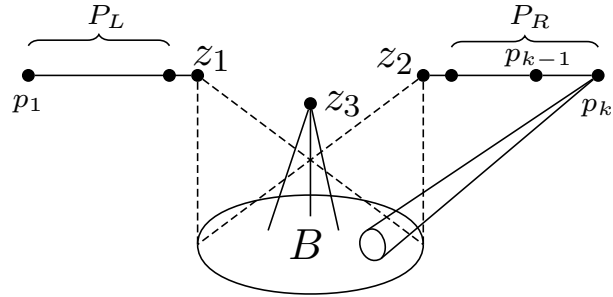

\sta{$\{z_1,z_2,z_3\}$ is constricted in $G_h$. \label{constricted}}

Suppose there is a tree $T$ in $G_h$ such that $z_1,z_2,z_3 \in V(T)$.
We may choose $T$ minimal with this property; then either $T$ is
a subdivision of $K_{1,3}$ and $z_1,z_2,z_3$ are leaves on $T$,
or $T$ is a path with ends in $\{z_1,z_2,z_3\}$.  Since
$G_h \setminus \{z_1,z_2,z_3\}$ is anticomplete to $P(h)$, in both cases $T \cup P(h)$ contains
$S_{t,t,t}$, a contradiction. This proves~\eqref{constricted}.
\\
\\
By Theorem~\ref{stripdecomp}, there is a graph $H(h)$ such that
$(G_h, \{z_1,z_2,z_3\})$ admits a
faithful extended strip decomposition $\eta_h$ with pattern $H(h)$.

\sta{There exists an atom $A(h)$ of $\eta_h$ such that $w(A(h))>\frac{1}{2}$.
\label{bigatom}}

Suppose not.
Then by Lemma~\ref{lem:bigatom} applied to $G_h$ and $w$ (with $D = G_h$),
we deduce that there exists $Y_h \subseteq G_h$
with $|Y_h| \leq d-22t$, such that $N[Y_h]$ is a $w$-balanced separator in
$G_h$. It follows from the definition of $G_h$
that $N[Y_h \cup Y_1 \cup Y_2 \cup P(h_0) \cup P(h)]$ is a $w$-balanced separator in $G$,
which is a contradiction since $|Y_1 \cup Y_2 \cup P(h_0) \cup P(h)| < 22t$. This proves~\eqref{bigatom}.
\\
\\
Let $A(h)$ be as in \eqref{bigatom}.

\sta{At least one of the sets $P_L(h) \cap A(h)$, $P_R(h) \cap A(h)$
is empty. \label{notbothsides}}

Suppose not. Let $Q_1=z_1 \dd p_{i-2t-2} \dd P_L(h)$. Since $N[P]$ is a
$w$-balanced separator, there exists $m \in N(p_k) \cap B$;
let $Q$ be a path from $h$ to $m$ with $Q^* \subseteq B$.
Then $F=z_2 \dd p_{i+2t+3} \dd P_R(h) \dd p_k \dd m \dd Q \dd h$
is a path. Since $P_R(h) \cap A(h) \neq \emptyset$,
there exists $f \in A(h) \cap F$. It follows from the last bullet of the definition of an
extended strip decomposition that $z_2, z_3$ do not belong to any atoms, and so $f \not \in \{z_2,z_3\}$. Let $Q_2'$ be the subpath of $F$ from $z_2$ to $f$, and let $Q_3'$ be the subpath of $F$
from $z_3$ to $f$ (see  Figure~\ref{fig: notbothsides}).
Let $Q_2=Q_2' \setminus f$ and $Q_3=Q_3' \setminus f$.
Then $Q_1,Q_2,Q_3$ are pairwise disjoint and anticomplete to each other;
$z_i$ is an end of $Q_i$, and  $Q_i \cap N[A] \neq \emptyset$
for every $i \in \{1,2,3\}$, contrary to Lemma~\ref{threepaths}.
This proves~\eqref{notbothsides}.
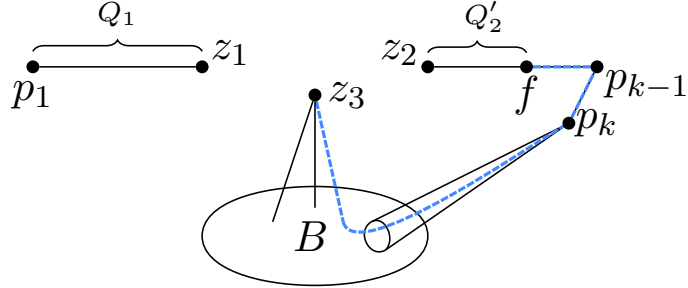
\begin{figure}[ht]
    \centering
    \scalebox{1.5}{\input{./figures/d_separable_notbothsides.tikz} }
    \caption{Visualization of $Q_1,Q_2'$ and $Q_3'$ (in blue).}
    \label{fig: notbothsides}
\end{figure}
\\
\\
Let $\delta(h)$ be the boundary of $A(h)$ in $\eta_h$ and let $\gamma(h)=\delta(h) \cap P$.

\sta{$|\gamma(h)| \leq 9$. \label{pathboundary}}

By Lemma~\ref{atomboundary} there exists $\Delta \subseteq G_h$ with
$|\Delta| \leq 3$ such that $\delta(h) \subseteq N[\Delta]$.
Since $N(P) \cap G_h \subseteq H_3$, for every $v \in \Delta$,
$|N[v] \cap P| \leq 3$. Consequently,
$|\delta(h) \cap P| \leq |N[\Delta] \cap P| \leq 9$, and \eqref{pathboundary} follows.
\\
\\
\sta{Let $Z \subseteq V(G)$ with $Y_1 \cup Y_2 \cup P(h_0) \cup P(h) \subseteq Z$  and such that $\delta(h) \subseteq N[Z]$. Let
$D \subseteq G \setminus N[Z]$ be connected with $w(D)>\frac{1}{2}$. Then
$D \subseteq A(h)$ and  there exists $v \in D \cap N(B)$.
\label{meet}}

Since $w(B)>\frac{1}{2}$, it follows that $B \cap D \neq\emptyset$.
Similarly, since $w(A(h))>\frac{1}{2}$, $A(h) \cap D \neq \emptyset$.
Since $Y_1 \cup Y_2 \cup P(h_0) \cup P(h) \subseteq Z$ and since $\delta(h) \subseteq N[Z]$, it follows that
$N_G(A(h)) \subseteq N[Z]$.
We deduce that  $D \subseteq A(h)$.
Since $p_k \in Y_1$,
and since $N[P]$ is a balanced separator in $G$, it follows that
$D \setminus B \neq \emptyset$. Since $D$ is connected, there exists
$v \in D \setminus B$ with a neighbor in $B$.
This proves~\eqref{meet}.
\\
\\
In view of \eqref{H3}, let $x'$ be minimum such that there exists $h_0'\in H_3$ with $N(h_0') \cap P=\{p_{x'},p_{x'+1}\}$; we refer to this as the {\em  minimality of $h_0'$}.

    \sta{$P_R(h_0') \cap A(h_0') \neq \emptyset$. \label{startright}}

Suppose that $P_R(h_0') \cap A(h_0') =\emptyset$.  Let $Y_3$ be a core of $\delta(h_0')$; choose $Y_3$ with
$|Y_3|$ minimum. By Lemma~\ref{atomboundary} $|Y_3| \leq 3$. Let $Y_4=\gamma(h_0')$; by \eqref{pathboundary} $|Y_4| \leq 9$.

Let $Z = Y_1 \cup Y_2 \cup Y_3 \cup Y_4 \cup P(h_0)\cup P(h_0')$.
We claim that $N[Z]$ is a balanced separator in $G$.
Suppose not, and let $D$ be a component of $G \setminus N[Z]$ with
$w(D) > \frac{1}{2}$. By \eqref{meet},  $D \subseteq A(h_0')$ and
there exists
$v \in D \setminus B$ with a neighbor in $B$. Then $v \in N[P]$;
let $v' \in P$ be a neighbor of $v$.
Since $D \subseteq A(h_0')$, we deduce that $v \in A(h_0')$,
and so $v' \in N[A(h_0')] \cap P$. Since
$Y_1 \cup Y_2 \cup P(h_0') \subseteq Z$, and by the minimality of
$h_0'$, we deduce that $v' \in P_R(h_0')$.
Since $P_R(h_0') \cap A(h_0') = \emptyset$, we conclude that
$v' \in N_{G_{h_0'}}(A(h_0')) \cap P \subseteq \delta(h_0') \cap P=\gamma(h_0')$.
But then $v' \in Y_4$, and so $v \in N[Z]$, contrary to the fact that
$v \in D$. This proves~\eqref{startright}.
\\
\\
\sta{$P_L(h_1) \cap A(h_1) \neq \emptyset$. \label{endleft}}

The proof is similar to the proof of \eqref{startright}.
Suppose that $P_L(h_1) \cap A(h_1) =\emptyset$.
Let $Y_3$ be a core of $\delta(h_1)$; choose $Y_3$ with
$|Y_3|$ minimum. By Lemma~\ref{atomboundary} $|Y_3| \leq 3$. Let $Y_4=\gamma(h_1)$; by \eqref{pathboundary} $|Y_4| \leq 9$.

Let $Z=Y_1 \cup Y_2 \cup P(h_0)\cup Y_3 \cup Y_4 \cup P(h_1)$.
We claim that $N[Z]$ is a balanced separator in
$G$.
Suppose not, and let $D$ be a component of $G \setminus N[Z]$ with
$w(D) > \frac{1}{2}$. By \eqref{meet},
$D \subseteq A(h_1)$ and there exists
$v \in D \setminus B$ with a neighbor in $B$. Then $v \in N[P]$;
let $v' \in P$ be a neighbor of $v$.
Since $D \subseteq A(h_1)$, we deduce that $v \in A(h_1)$,
and so $v' \in N[A(h_1)] \cap P$.
Since $Y_1 \cup Y_2 \cup P(h_0) \cup P(h_1) \subseteq Z$, and by the maximality of
$h_1$, we deduce that $v' \in P_L(h_1)$.
Since $P_L(h_1) \cap A(h_1) = \emptyset$, we conclude that
$v' \in N_{G_{h_1}}(A(h_1)) \cap P \subseteq \delta(h_1) \cap P=\gamma(h_1)$.
But then $v' \in Y_4$, and so $v \in N[Z]$, contrary to the fact that
$v \in D$. This proves~\eqref{endleft}.
\\
\\
By \eqref{notbothsides} and \eqref{startright} $P_L(h_0') \cap A(h_0') = \emptyset$.
In view of this, let $i$ be maximum such that there exist $h_2 \in H_3$ with
$P_L(h_2) \cap A(h_2) = \emptyset$  and $N(h_2) \cap P=\{p_i,p_{i+1}\}$.
By Lemma~\ref{atomboundary} there is a core $Z_2$ for $\delta(h_2)$
with $|Z_2| \leq 3$.
By \eqref{endleft}, $i<y$, and therefore, since $y\in H_3$, there exists $j>i$ such that there exists $h_3 \in H_3$ with $N(h_3) \cap P =\{p_j,p_{j+1}\}$; we may assume that $j$ is chosen minimum with this property.
Then
$P_L(h_3)  \cap A(h_3) \neq  \emptyset$, and therefore  by
\eqref{notbothsides} $P_R(h_3)  \cap A(h_3)  = \emptyset$.
By Lemma~\ref{atomboundary} there is a core $Z_3$ for $\delta(h_3)$
with $|Z_3| \leq 3$. Let
$Z=Y_1 \cup Y_2 \cup P(h_0) \cup Z_2 \cup Z_3 \cup P(h_2) \cup P(h_3) \cup \gamma(h_2) \cup \gamma(h_3)$. Then $|Z| < d$. To complete the proof, it remains to show that
$N[Z]$ is a $w$-balanced separator in $G$.

Suppose not, and let $D$ be a component of $G \setminus N[Z]$ with
$w(D) > \frac{1}{2}$. By \eqref{meet},
$D \subseteq A(h_2) \cap A(h_3)$ and there exists
$v \in D \setminus B$ with a neighbor in $B$. Then $v \in N[P]$, and therefore
$v \in H_3$; let $v' \in P$ be a neighbor of $v$.
Since $D \subseteq A(h_2) \cap A(h_3)$, we deduce that $v \in A(h_2) \cap A(h_3)$, and so $v' \in N[A(h_2)] \cap N[A(h_3)] \cap P$.

It follows from the choice of $h_3$ that
$N(D) \cap P \subseteq P_L(h_2) \cup P_R(h_3) \cup P(h_2) \cup P(h_3)$,
and therefore $v' \in  P_L(h_2) \cup P_R(h_3) \cup P(h_2) \cup P(h_3)$.
Since $P_L(h_2) \cap A(h_2) = \emptyset$ and $P_R(h_3) \cap A(h_3)=\emptyset$,
and $Y_1 \cup Y_2 \cup P(h_0)\cup P(h_2) \cup P(h_3)  \subseteq Z$,
we deduce that $v' \not \in A(h_2) \cap A(h_3)$, and therefore either $v' \in N_{G_{h_2}}(A(h_2)) \cap P\subseteq \delta(h_2) \cap P=\gamma(h_2)$ or $v' \in N_{G_{h_3}}(A(h_3)) \cap P \subseteq \delta(h_3) \cap P=\gamma(h_3)$.
But then $v' \in \gamma(h_2) \cup \gamma(h_3)$, and so $v \in N[Z]$, contrary to the fact that $v \in D$.
\end{proof}

\section{Large stable sets in neighborhoods}
\label{sec:alpha_of_neighborhood}
For positive integers $a,b$ let $R(a,b)$ be the smallest integer $R$ such that every graph on $R$ vertices  contains either a stable set of size $a$ or a clique of size $b$.
The $2$-subdivision of a graph $H$, denoted by $H^{(2)}$, is the graph obtained by subdividing each edge in $H$ twice (so each edge of $H$ is replaced by a three-edge path). In particular, $K_{\gamma}^{(2)}$ is the $2$-subdivision of the complete graph on $\gamma$ vertices.

Consider a  graph $G$ that is $K_{\gamma}^{(2)}$-free and $K_{t,t}$-free. We will show that for every subset $Y \subseteq V(G)$ with large independence number, the set of vertices $Z$ with reasonably large independent sets in their neighborhoods in $Y$ has small independence number, as follows:

\begin{theorem}\label{few big independent neighborhoods in big independent set}
Let $C,\gamma, t \in \nat$ such that $C,\gamma \geq 2$, and let $G$ be a  $\{K_{\gamma}^{(2)},K_{t,t}\}$-free graph. Let $Y \subseteq V(G)$. Define
\[Z = \left\{z \in V(G)\colon \alpha(N(z) \cap Y) \geq \frac{\alpha(Y)}{C}\right\}\,.\]
Then
\[\min(\alpha(Y),\alpha(Z)) \leq (512C)^{\gamma^{2t}}\,.\]
\end{theorem}

A pair $(A,B)$ is a {\em $K_{s,t}$ in $G$} if $|A| = s$, $|B| = t$ and $A,B$ are disjoint stable sets in $G$ that are complete to each other. (So $A \cup B$ is a  $K_{s,t}$ in $G$.)  Let $Y \subseteq V(G)$. We say that $(A,B)$ is a {\em $K_{s,t}$ in $G$ with respect to $Y$} if $(A,B)$ is a $K_{s,t}$ in $G$ and $B \subseteq Y$. We say that $G$ is {\em $K_{s,t}$-free with respect to $Y$} if there is no $K_{s,t}$ with respect to $Y$. We start with a lemma.

\begin{lemma}\label{alpha-lemma-technical}
Let $G$ be a graph and let $C,\gamma, s,t \in \NN$ such that  $1 \leq s \leq t$ and $C,\gamma \geq 2$. For $i \in \{1,\ldots,s\}$, define
\[c_i = (8^{s-i} \cdot C)^{\gamma^{s-i}}\,,\]
\[f(i) = f_{s,t,\gamma,C}(i) = t(C \cdot 8^s)^{2i\gamma^s}\,.\]
Let $Y \subseteq V(G)$ and let
\[Z = \left\{z \in V(G)\colon \alpha(N(z) \cap Y) \geq \frac{\alpha(Y)}{c_i}\right\}\,.\]
Assume that  the following two conditions hold:
\begin{itemize}
\item[(a)] $G$ is $K_\gamma^{(2)}$-free and $K_{i,t}$-free with respect to $Y$.
\item[(b)] $\alpha(Y) > f(i)$.
\end{itemize}
Then $\alpha(Z) < f(i)$.
\end{lemma}

\begin{proof}
We proceed by induction on $i$. Let $i = 1$. Suppose $Y \subseteq G$ satisfies conditions (a) and (b) above. So $\alpha(Y) > f(1) > c_1 t$. Since $G$ is $K_{1,t}$-free with respect to $Y$, it follows that for every $z \in V(G)$, $\alpha(N(z) \cap Y) < t < \frac{\alpha(Y)}{c_1}$. We deduce  that $Z$ is empty, so $\alpha(Z) < 1 < f(1)$, as required.

Now let $i \in \{2,\ldots,s\}$ and assume inductively that the result holds for $i-1$. Assume that  $Y$ satisfies conditions (a) and (b) above. We need to show that $\alpha(Z) < f(i)$.

Suppose not. By taking a subset of $Z$ if necessary, we may assume that $\alpha(Z) = |Z| = f(i)$ (and in particular $Z$ is an independent set). For each $z \in Z$, let $J''(z)$ be an independent set of size $\lceil \frac{\alpha(Y)}{c_i} \rceil$ in $N(z) \cap Y$.  Write
\[d=\frac{1}{c_{i-1}} \left\lceil \frac{\alpha(Y)}{c_i}\right\rceil\,.\]

We now perform some cleaning steps, modifying $Z$ and $J''(z)$.

\sta{Let $z \in Z$, and let $F(z)$ be the set of all vertices $z'\in V(G) \setminus N[z]$ such that $|N(z') \cap J''(z)| \geq d$. Then $\alpha(F(z)) < f(i-1)$. \label{degree}}

Let $G'= G \setminus (N[z]\setminus J''(z))$ and let $Y' = J''(z)$.  If $(A,B)$ is a $K_{{i-1},t}$ with respect to $Y'$ in $G'$,
then $(A \cup \{z\},B)$ is a $K_{i,t}$ with respect to $Y$ in $G$,
and therefore,  $Y'$ satisfies condition (a) with $i-1$ in $G'$.
Further, for every $i \geq 2$, we have that
\[|Y'|=\left\lceil \frac{\alpha(Y)}{c_i} \right\rceil > \frac{f(i)}{c_i} \geq f(i-1)\,,\]
and so $Y'$  satisfies condition (b) with $i-1$ in $G'$.
Let $Z' = \{z \in V(G')\colon \alpha(N(z) \cap Y') \geq d\}$.
It follows inductively that $\alpha(Z') < f(i-1)$. Since $F(z) \subseteq Z'$,
\eqref{degree} follows.
\\
\\
\sta{\label{claim: cleaning 1}There exists $I \subseteq Z$ with $|I| = (8c_i)^{\gamma-1}$ such that for all distinct $z,z' \in I$, $|N(z) \cap J''(z')| < d$.}

Let $K$ be a directed graph with vertex set $Z$ where $(z_1,z_2) \in E(K)$ if and only if  $|N(z_1) \cap J''(z_2)|  \geq d$.
Since $Z$ is a stable set, it follows from \eqref{degree} that for all $z \in Z$, the indegree of $z$ in $K$ is less than $f(i-1)$.
By an easy degeneracy argument (see, for example, Lemma 5.2 of \cite{TW10})
$K$ contains an independent set $I$ of size at least \[\frac{f(i)}{2f(i-1)} = \frac{1}{2}(C\cdot8^s)^{2\gamma^s} =
\frac{1}{2}(C\cdot 8^{s-i})^{\gamma^{s-i}\cdot2\gamma^{i}} 8^{2i\gamma^s} = \frac{1}{2} c_i^{2\gamma^i} 8^{2i\gamma^s}
\geq \frac{1}{2} (8c_i)^{\gamma} \geq (8c_i)^{\gamma-1}\,.\]
This proves (\ref{claim: cleaning 1}).
\\
\\
Let $I$ be as in \eqref{claim: cleaning 1}.

\sta{\label{claim: cleaning 2} For every $z \in I$ there exists $J'(z) \subseteq J''(z)$ with  $|J'(z)| \geq |J''(z)| - |I| d$ such that for all distinct $z,z' \in I$, $J'(z) \cap J'(z') = \emptyset$.}

Let  $z \in I$.  It follows from the definition of $I$ that
for every  $z' \in I$ with $z' \neq z$, we have  that $|N(z') \cap J''(z)| < d$. Define   \[J'(z) = J''(z) \setminus \bigcup_{z' \in I \setminus \{z\}} N(z')\,.\]
It then follows that $|J'(z)| \geq |J''(z)| - |I|d$ for all $z \in I$, and the sets $J'(z)$ are pairwise disjoint. This proves (\ref{claim: cleaning 2}).
\\
\\
\sta{\label{claim: J size} For every $z \in I$ there  exists $J(z) \subseteq J'(z)$ with $|J(z)| \geq |J''(z)| - |I| \cdot \left(f(i-1) + d\right)$ such that for every vertex $v \in \bigcup_{z' \in I} J(z')$ we have
$|N(v) \cap J(z)|<d$.}

Let $z \in I$. Let $F'(z)=\bigcup_{z' \in I \setminus \{z\}} \left(J'(z) \cap F(z')\right)$.
By \eqref{degree} $\alpha(F'(z)) <|I|f(i-1)$. Define
\[J(z)=J'(z) \setminus F'(z)\,.\]
Now $|J(z)| \geq |J''(z)| - |I| \cdot \left(f(i-1) + d\right)$  as required.
This proves~\eqref{claim: J size}.
\\
\\
By known upper bounds on Ramsey numbers \cite{ramsey}, $|I| \geq R(4c_i, \gamma)$.

Let $\Gamma$ be a graph with vertex set $I$, where two distinct vertices $z, z' \in I$ are adjacent
if and only if there is a matching of size at least $m = 2\gamma^2 d$ between $J(z)$ and $J(z')$. Since by \eqref{claim: cleaning 1} and  Ramsey   theorem   \cite{ramsey}  $|I| \geq R(4c_i,\gamma)$,  it follows that
$\Gamma$ contains a stable set of size $4c_i$ or a clique of size $\gamma$.
We will finish the proof by showing that both cases lead to a contradiction.

\sta{\label{claim: no red clique}$\Gamma$ has no  stable set of size $4c_i$.}

 Suppose there is a stable set $S$ in $\Gamma$ of size $4c_i$.
 Let $z,z' \in S$ be distinct. Since there is no   matching of size $m$ between $J(z)$ and $J(z')$,
 by K\"{o}nig's Theorem \cite{Konig} there exists $X_{zz'} \subseteq J(z) \cup J(z')$ with $|X_{zz'}|<m$ such that $J(z) \setminus X_{zz'}$ is anticomplete to
 $J(z') \setminus X_{zz'}$. Let $X=\bigcup_{z \neq z' \in S}X_{zz'}$.
Then $J=\bigcup_{z \in S} J(z) \setminus X$ is a stable set and
\[|X| \leq (4c_i)^2m = 16c_i^2\frac{2\gamma^2}{c_{i-1}} \left\lceil \frac{\alpha(Y)}{c_i}\right\rceil\,.\]
By \eqref{claim: J size}
\[|J| \geq 4c_i(\left\lceil \frac{\alpha(Y)}{c_i}\right\rceil - |I|\left(f(i-1) + \frac{1}{c_{i-1}} \left\lceil \frac{\alpha(Y)}{c_i}\right\rceil\right))-|X| \geq \]
\[4c_i\left(\left\lceil \frac{\alpha(Y)}{c_i}\right\rceil - |I|\left(f(i-1) + \frac{1}{c_{i-1}} \left\lceil \frac{\alpha(Y)}{c_i}\right\rceil\right) - \frac{8\gamma^2 c_i}{c_{i-1}} \left\lceil \frac{\alpha(Y)}{c_i}\right\rceil \right)\,.\]

Since $J \subseteq Y$, it follows that  $|J| \leq \alpha(Y)$.
Simplifying,
\[\left\lceil\frac{\alpha(Y)}{c_i}\right\rceil \left(4c_i - \frac{4c_i|I|}{c_{i-1}} - \frac{8\gamma^2 c_i}{c_{i-1}}\right) \leq 4c_i|I|f(i-1) + \alpha(Y)\,.\]

Note that
\[\frac{c_{i-1}}{4} = \frac{(8^{s-i+1}C)^{\gamma^{s-i+1}}}{4} = \frac{(8^{s-i}C)^{\gamma^{s-i+1}} 8^{\gamma^{s-i+1}}}{4} \geq \frac{(8c_i)^{\gamma}}{4} \geq (8c_i)^{\gamma-1} = |I|\]
and
\[c_{i-1} = (8^{s-i+1}C)^{\gamma^{s-i+1}} \geq (8C)^{\gamma} \geq 16^\gamma \geq 8\gamma^2\,.\]
Therefore, $\frac{4|I|}{c_{i-1}} \leq 1$ and
$\frac{8\gamma^2}{c_{i-1}} \leq 1$.
It follows that
\[\left\lceil\frac{\alpha(Y)}{c_i}\right\rceil \left(4c_i - c_i - c_i\right) \leq 4c_i|I|f(i-1) + \alpha(Y)\,.\]
We may drop the ceiling signs and rearrange to obtain \[\alpha(Y) \leq 4 c_i |I| f(i-1) = \frac{1}{2} (8c_i)^{\gamma} f(i-1) \leq f(i)\,.\]
This is a contradiction as $\alpha(Y) > f(i)$, which proves (\ref{claim: no red clique}).

\sta{\label{claim: no blue clique}$\Gamma$ has no clique of size $\gamma$.}
 Suppose there exists clique $K$ of size $\gamma$ in $\Gamma$. We will obtain a contradiction by constructing a $K_{\gamma}^{(2)}$ in $G$.
Write $K= \{z_1,\ldots,z_{\gamma}\}$. For each $i,j \in [\gamma]$ with $i < j$, we will find a vertex $u \in J(z_i)$ and a vertex $v \in J(z_j)$ such that $u$ is adjacent to $v$,  so we may connect $z_i$ to $z_j$ via the path $z_i \dd u \dd v \dd  z_j$. We will ensure that the choice of $u$ and $v$ for each $i$ and $j$ is such that the resulting graph is a  $K_{\gamma}^{(2)}$ in $G$.

Loop through the pairs $(i,j) \in [\gamma]^2$ with $i < j$. Fix some $i,j$. We have previously connected fewer than $\binom{\gamma}{2}$ pairs $(i',j')$. For each such previously connected pair $(i',j')$, we chose  $u' \in J(z_{i'})$ and $v' \in J(z_{j'})$. It follows that so far  we have used fewer than $\gamma(\gamma-1)$ vertices from $\bigcup_{i' \in K} J(z_{i'})$; denote the set of these previously used vertices by $X$.
By $\eqref{claim: J size}$,
for each  $v \in X$,
$|N(v) \cap J(z_i)|<d$ and $|N(v) \cap J(z_j)|<d$.
Since $|X| < \gamma(\gamma-1)$,
and since there is a matching of size $m$ between $J(z_i)$ and $J(z_j)$,
there exist  $u \in J(z_i) \setminus X$ and $v \in J(z_j) \setminus X$
such that $u$ is adjacent to $v$. Connect $z_i$ to $z_j$ using $u$ and $v$ via the path $ z_i \dd u \dd v \dd z_j$. This gives us a $K_{\gamma}^{(2)}$, a contradiction that proves (\ref{claim: no blue clique}).
\end{proof}

\begin{proof}[Proof of Theorem~\ref{few big independent neighborhoods in big independent set}]
Let $i = s = t$, $Z = \left\{z \in V(G)\colon \alpha(N(z) \cap Y) \geq \frac{\alpha(Y)}{C}\right\}$ and let $f(i)$ be defined as in Lemma~\ref{alpha-lemma-technical}. Note that
\begin{align*}
t(C\cdot8^t)^{2t\gamma^{t}}
&\leq (2C\cdot8^t)^{2t\gamma^{t}}\\
&\leq (2C)^{2t\gamma^t} \cdot 2^{6t^2 \gamma^t}\\
&\leq (2C)^{\gamma^{2t}} \cdot 2^{8 \gamma^{2t}}\\
&= (512C)^{\gamma^{2t}}\,.
\end{align*}

We may assume
that $\alpha(Y)>(512C)^{\gamma^{2t}} \geq t(C\cdot8^t)^{2t\gamma^{t}} = f(t)$ as otherwise, there is nothing to show.
Since $G$ is $\{K_{\gamma}^{(2)},K_{t,t}\}$-free and $\alpha(Y) >f(t)$, Lemma~\ref{alpha-lemma-technical} applies. It follows that $\alpha(Z) < f(t) \leq (512C)^{\gamma^{2t}}$.
\end{proof}

\section{The layered sets argument}
\label{sec:layeredsets}
We start with a few more definitions. Let $\mathcal{C}$ be a class of graphs and
let $k$ be a positive integer. We say that $\mathcal{C}$  is {\em $k$-breakable}
if $\mathcal{C}$ is closed under taking induced subgraphs and every
graph in $\mathcal{C}$ is $k$-breakable.
Next we define a more refined version of balanced separators.
Let $\eps\in(0,1]$, $G$ be a graph and $w$ be a weight function on $G$. Let $(S, C)$ be a pair of subsets of $V(G)$ and let $B$ be a component of $G\setminus C$ with maximum weight.
The pair $(S,C)$ is said to be a {\em $(w,\eps)$-boosted separator} of $G$ if $w(B)\leq 1/2$ or if $S\cap B$ is a $(w,\eps)$-balanced separator of $B$.
We call $C$ the {\em boosting set} of $(S,C)$.
A set $X\subseteq V(G)$ is said to be a {\em core} of $(S,C)$ if $S\subseteq N[X]$. Let $k\in \nat$, $f\colon \nat\rightarrow\reel_{\geq0}$. A graph $G$ is said to be {\em $(k,f,\eps)$-breakable} if for every weight function $w$, there exists a $(w,\eps)$-boosted separator $(S,C)$ with a core of size at most $k$ and $\alpha(C)\leq f(|V(G)|)$.  A class $\mathcal{C}$ of graphs is
$(k,f,\eps)$-breakable if $\mathcal{C}$ is closed under taking induced subgraphs and every graph in $\mathcal{C}$ is $(k,f, \eps)$-breakable.
If $f(n)$ is a constant function
with $f(n)=c$, we say that $G$  (or $\mathcal{C}$) is  {\em $(k,c,\eps)$-breakable}.

The goal of this section is to prove the following:
\begin{theorem}\label{thm : boosted construction result}
For all positive integers $k,\gamma, t,\lambda$ there exists an integer $c=c(k,\gamma, t, \lambda)$
with the following properties.  Let
$f\colon \nat\rightarrow\reel_{\geq0}$ and $\eps\in(0,1]$. Let $\mathcal{C}$ be a
$(k,f,\eps)$-breakable class of graphs, and let $G \in \mathcal{C}$ be
a $\{K_\gamma^{(2)}, K_{t,t}\}$-free  graph on $n\geq2$ vertices. Let $w$ be a weight function on $G$.
Then there exist $C,S_1,\dots,S_{\lceil \log n \rceil} \subseteq V(G)$ where
\begin{enumerate}[(I)]
\item \label{small alpha} $\alpha(C) \leq c (f(n) \log n + \log^2 n)$,
\item \label{boosted sep} for all $i$, $(S_i,C)$ is a $(w,\eps)$-boosted separator of $G$ with a core $X_i$ of size less than $k$,
\item \label{almost disjoint} if there is a component $G'$ of $G\setminus C$ with $w(G')>\frac{1}{2}$, then no vertex of $G'$ is contained in more than $\frac{3\log n}{k\lambda}$ of the sets $S_1,\ldots,S_{\lceil \log n \rceil}$, and
\item \label{almost disjoint cores} for every $j\in \{1,\ldots, \lceil \log n \rceil\}$, no vertex of $X_j$ is contained in more than $\frac{3\log n}{k\lambda}$ of the sets $S_i$ with $i<j$.
\end{enumerate}

\end{theorem}
To prove this theorem, we give an algorithm that outputs a set of $(w,\eps)$-boosted separators and prove that they have the required proprieties.

The following definition is key  in the  proof.
Let $G$ be a graph and $\mathcal{F}$ be a family of sets with ground set $V(G)$. Then the {\em $i^{th}$ layer} of $\mathcal{F}$ in $G$ is
\[L^i(G, \mathcal{F}) =  \{v \in V(G)\colon  |\{F \colon  F\in \mathcal{F}, v \in F\}| \geq i\}\,.\]

\begin{proof}[Proof of Theorem \ref{thm : boosted construction result}]
We may assume that $n\geq 2^{2k\lambda}$ as otherwise the theorem holds trivially with $c = 2^{2k\lambda}$, $C =G$ and $S_1=\dots =S_{\lceil\log n\rceil} =\mt$.

Let $d = k 2^{\lambda k}$ and $T=(512d)^{\gamma^{2t}}$.
Consider Algorithm \ref{algo: Layered sets}.

 \begin{algorithm}[h]
\begin{algorithmic}[1]
\State  $G_0 = G$
    \For{$j = 1,\dots, \lceil \log n \rceil$}
	    \State {$(S_j,Y_j) \coloneqq $ a $(w,\eps)$-boosted separator of $G_{j-1}$ where $S_j=N[X_j]$, $|X_j| < k$, and $\alpha(Y_j) \leq f(n)$}\\
        \Comment{\textit{Such a pair $(S_j,Y_j)$ exists since $\mathcal{C}$ is $(k,f,\eps)$-breakable.}}
        \For{$i = 1,\dots,j+1$}
            \State $L_j^i \coloneqq  L^i(G_{j-1},\set{S_1,\dots,S_j})$
            \State $Z_j^i \coloneqq  \set{v \in V(G_{j-1})\colon \alpha(N(v) \cap L^i_j) \geq \frac{\alpha(L^j_i)}{d} }$
            \If {$\alpha(L_j^i) \leq T$}
                \State $C_j^i \coloneqq  L_j^i$
            \Else
                \State $C_j^i \coloneqq  Z_j^i$
            \EndIf
        \EndFor
        \State $C_j = Y_j \cup \bigcup_{i=1}^j C_j^i$
        \State {$G_{j} \coloneqq $ a maximum weight component of $G_{j-1}\setminus C_j$}
	\EndFor
\State $C = \bigcup_{j=1}^{\log n}C_j$
\State \Return $(S_1,C),\dots,(S_{\lceil \log n \rceil},C)$
\end{algorithmic}
\caption{Layered Sets Algorithm}
\label{algo: Layered sets}
\end{algorithm}

Since for all $j\in \{1,\ldots,\lceil \log n\rceil\}$, $(S_j,Y_j)$ is a $(w,\eps)$-boosted separator of $G_{j-1}$, and $G_{j}$ is a component of $G_{j-1}\setminus C_{j}$, and
$Y_j\subseteq C_j\subseteq C$,
it follows that each $(S_j,C)$ is a $(w,\eps)$-boosted separator of $G$, and so \ref{boosted sep} holds.

\sta{\label{claim: small Level boosted}For all $1 \leq i \leq j\le \lceil \log n\rceil$,
\[\alpha(L_j^i) \leq \frac{2^{j-1}}{2^{\lambda k(i-1)}} n\,.\]}

 We proceed by induction on $j$.
Whenever $i=1$, the bound is trivial. When $j = 1$, we must also have that $i = 1$. So the base case holds. So now assume $2 \leq i \leq j$.
If $i = j$, then $L^{i}_{j-1} = \emptyset$.
If $i < j$, by the induction hypothesis, $\alpha(L^{i}_{j-1}) \leq \frac{2^{j-2}}{2^{\lambda k(i-1)}} n$. Since $L_j^i \subseteq L_{j-1}^i \cup (L^{i-1}_{j-1} \cap S_j)$, it suffices to show that $\alpha(L^{i-1}_{j-1} \cap S_j) \leq \frac{2^{j-2}}{2^{\lambda k(i-1)}}n$ as well. There are two possible cases.

\begin{itemize}
\item {\em Case 1: $\alpha(L^{i-1}_{j-1}) \leq T$}. In this case  $C^{i-1}_{j-1} = L^{i-1}_{j-1} \subseteq C_{j-1}$ is removed from $G_{j-2}$ when forming $G_{j-1}$, and so none of the vertices in $L^{i-1}_{j-1}$ belong to $G_{j-1}$.
Therefore, $L^{i-1}_{j-1} \cap S_j = \mt$, so we are done.
\item {\em Case 2: $\alpha(L^{i-1}_{j-1}) > T$}.
In this case we have that $C^{i-1}_{j-1} = Z^{i-1}_{j-1} \subseteq C_{j-1}$ is removed from $G_{j-2}$ when forming $G_{j-1}$.
It follows that for every vertex $v \in G_{j-1}$,
\[\alpha(N(v) \cap L^{i-1}_{j-1}) < \frac{\alpha(L^{i-1}_{j-1})}{d} \leq \frac{2^{j-2}}{d 2^{\lambda k(i-2)}} n.\]

Since $L^{i-1}_{j-1} \cap S_j \subseteq (X_j \cap L^{i-1}_{j-1}) \cup (\bigcup_{v \in X_j} N(v) \cap L^{i-1}_{j-1})$, we deduce
\[ \alpha(L^{i-1}_{j-1} \cap S_j) \leq \sum_{v \in X_j} \max(\alpha(N(v) \cap L^{i-1}_{j-1}),1)\,.\]
Therefore
\[\alpha(L^{i-1}_{j-1} \cap S_j)  \leq k \frac{\alpha(L^{i-1}_{j-1})}{d}  \leq \frac{2^{j-2}}{2^{\lambda k(i-1)}} n\,.\]
\end{itemize}
This completes the induction and proves \eqref{claim: small Level boosted}.
To prove that \ref{almost disjoint} holds,  let
 $i > 1 + \frac{1+2\log n}{k\lambda}$ and $j=\lceil \log n \rceil$. By \eqref{claim: small Level boosted}, $\alpha(L^i_{j}) \leq \frac{2^{\log n}}{2^{\lambda k(i-1)}} n < 1$, and therefore $|L^i_{j}|=0$.
Since, $  \frac{3\log n}{k\lambda} \geq 1 + \frac{1+2\log n}{k\lambda}$, this proves that  \ref{almost disjoint} holds.

An immediate consequence of \eqref{claim: small Level boosted} and the same calculation as the proof of \ref{almost disjoint} is that for every $j$ no vertex of $G_j$ is in more than $\min\set{j,\frac{3\log n}{k\lambda}}$ of the sets $S_1,\dots, S_j$, which proves \ref{almost disjoint cores}.

Finally, we have
    \begin{align*}
        \alpha(C) = \alpha\left(\bigcup_{j=1}^{\log n}C_j\right)
        &= \alpha\left( \bigcup_{j=1}^{\log n} Y_j \cup \bigcup_{j=1}^{\log n}\bigcup_{i=1}^j C^i_j\right)\\
        &\leq \sum_{j=1}^{\log n} \alpha(Y_j) + \sum_{j=1}^{\log n}\sum_{i=1}^j \alpha\left(C_j^i\right)\\
        &\leq f(n)\log n  + \log(n)^2 T\\
        & = f(n)\log n  + \log(n)^2 (512k 2^{\lambda k})^{\gamma^{2t}}\,,
\end{align*}
where the second inequality follows from  Theorem~\ref{few big independent neighborhoods in big independent set}.
Setting $c = (512k 2^{\lambda k})^{\gamma^{2t}}$ proves that \ref{small alpha} holds.

\end{proof}

\section{Improving the separators}
\label{sec: better seps}

\subsection{Improving the Separating Power}
\label{sec: improving separator}
The goal of this section is to ``boost'' the separators given by Theorem \ref{thm:domsep} to $(w,\eps)$-boosted separators without changing the size of the core and while introducing a boosting set with a somewhat small stability number. This, however, will be at the cost of forbidding $K_{t,t}$.
The idea of boosting the separators without increasing the size of the core is inspired by an argument of Gartland et al.~(\cite{GartlandLMPPR24}, Section~4).

The existence of $1/2$-balanced separators with small cores allows us to obtain $1/2^i$-balanced separators with relatively small cores.
\begin{lemma}[Analogue of Lemma 2 of \cite{gartland_independent_2020}]
\label{lemma : 1/4 separator}
Let $k\ge 2$ be an integer and let
 $\mathcal{C}$ be a $k$-breakable class of graphs.
    Let $G \in \mathcal{C}$ and let $w$ be a weight function on $G$.
    Then, for every positive integer $i$ there exists $X \subseteq V(G)$ with $|X| < 2^{i+1}(k-1)$,
    such that $N[X]$ is a   $(w,1/2^i)$-balanced separator of $G$.
\end{lemma}
\begin{proof}We proceed by induction on $i$. If $i=1$, the result follows from
the fact that $G$ is $k$-breakable. If $i>1$, the induction hypothesis gives us a set $X_0$ with  $|X_0|< 2^{i}(k-1)$ such that $N[X_0]$ is a
$(w,1/2^{i-1})$-balanced separator in $G$.  Let $\mathcal{D}$ be the set of
all components $D$ of  $G\setminus N[X_0]$ such that $w(D)\geq 1/2^i$.
Then $|\mathcal{D}| \leq 2^i$.
Let $D \in \mathcal{D}$ and let $w'(x)=2^{i-1}w(x)$ for all $x\in D$.
Then $w'(D)\leq 1$.
Since $\mathcal{C}$ is $k$-breakable, it follows that $D$ is $k$-breakable, and so
there exists $X(D) \subseteq D$ such that $|X(D)|\leq k-1$ and
$N[X(D)]$ is a
$(w',1/2)$-balanced separator in $D$.  Now let
$X=X_0 \cup \bigcup_{D \in \mathcal{D}}X(D)$.
Then
\[|X| \leq |X_0| + \sum_{D\in \mathcal{D}}|X(D)| < 2^{i}(k-1) + 2^i (k-1) = 2^{i+1}(k-1)\]
and $N[X]$ is a $(w,1/2^i)$-separator of $G$.
\end{proof}

\begin{corollary}
\label{coro: separators are boosted separators }
Let $k\ge 2$ be an integer.
Then, for every positive integer $i$, every $k$-breakable class of graphs is $\left(2^{i+1}(k-1),0,1/2^i\right)$-breakable.
\end{corollary}
\begin{proof}
    Let  $G \in \mathcal{C}$ and let $w$ be a weight function on $G$. Let $X \subseteq G$ with $|X|\leq 2^{i+1}(k-1)$ be a core of a $(w,1/2^i)$-separator guaranteed by Lemma \ref{lemma : 1/4 separator}.
    Then $(N[X],\mt)$ is a $(w,1/2^i)$-boosted separator of $G$.
\end{proof}

Let $G$ be a graph and $\eps\in(0,1]$.
We now introduce new terminology that will be useful in the next two lemmas and in which we will suppress the dependencies on $\eps$ to make the notation more concise.
Let $w$ be a weight function on $G$. Let $X \subseteq G$ be such that $|X| < k$ and $N[X]$ is a $(w, \frac{1}{2})$-balanced separator in $G$.
We will say that a connected component $B$ of $G$ is {\em big} if $w(B)>\eps$.
We denote by $\mathcal{B}(G)$ the set of big components of $G$.
Let $\beta \ge 0$, let $B$ be a big component of $G\setminus N[X]$, and let  $x\in X$.
We  call the pair $(x,B)$ {\em $(X,\beta)$-problematic} if $\alpha(N(x) \cap N(B)) \geq \beta$.
Let $\mathcal{P}(G,X,\beta)$ be the set of $(X,\frac{3\beta}{4})$-problematic pairs of $G$.
Let
\[W(G,X,\beta)=\sum_{(x,B)\in \mathcal{P}(G,X,\beta)} w(B)\,.\]

Let $\beta(G,X)$ be the maximum value of $\beta$ for which an $(X,\beta)$-problematic pair of $G$ exists (with $\beta(G,X) = 0$ if no big component of $G\setminus N[X]$ exists or $X = \emptyset$).

\begin{lemma} \label{lem : boosting one step} Let $k,t\in \nat$ with $k\ge 2$ and $\eps\in(0,1]$.
Let $\eta= c(16(k-1),3t+1,t,6t)$ be as in Theorem~\ref{thm : boosted construction result}.
Let $\mathcal{C}$ be a $k$-breakable class of graphs, and let
$G \in \mathcal{C}$ be  an $\set{S_{t,t,t},K_{t,t}}$-free graph on $n$ vertices, where $n\ge 2$.
Let $w$ be a weight function on $G$. Let $X \subseteq G$ be such that $|X| < k$ and $N[X]$ is a $(w, \frac{1}{2})$-balanced separator for $G$.
Let $\beta = \beta(G,X)> 4\eta\log^2n$ and let $\beta'\in [\beta,\frac{4}{3} \beta]$.
Then, there exists $C\subseteq V(G)\setminus X$, such that $\alpha(C) \leq \eta \log^2n$ and $W(G\setminus C,X ,\beta')\leq W(G, X ,\beta')-\eps$.
\end{lemma}
\begin{proof}
    Let $(x,B)$ be an $(X,\beta)$-problematic pair. Let  $I \subseteq N(x) \cap N(B)$
be a stable set of size $\beta$ (see Figure~\ref{fig:boosted separator}).
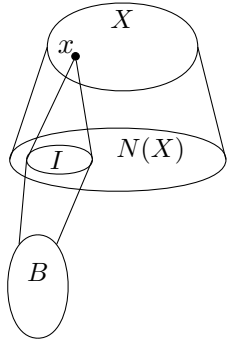
\begin{figure}[h]
    \centering
    \input{./figures/boosted_separator.tikz}
      \caption{Drawing to keep in mind for the proof of Lemma \ref{lem : boosting one step}.}
    \label{fig:boosted separator}
\end{figure}

Define a new  weight function $w'\colon V(G) \rightarrow [0,1]$ where $w'(v) = \frac{1}{|I|}$ if $v \in I$ and $w'(v)=0$ otherwise.
By Corollary~\ref{coro: separators are boosted separators }, $G$ is   $(16(k-1),0,1/8)$-breakable.
Note that since $G$ is $S_{t,t,t}$-free, $G$ is also $K_{3t+1}^{(2)}$-free.
Applying Theorem~\ref{thm : boosted construction result} to $G\setminus X$ with $\gamma = 3t+1$ and $\lambda=6t$, we obtain sets  $C,S_1,\dots,S_{\lceil \log n \rceil } \subseteq V(G)$ with  $\alpha(C) \leq \eta \log^2n$ where $(S_i,C)$ is a $(w',1/8)$-boosted separator of $G\setminus X$ with a core of size less than $16(k-1)$ for every $ i$.
Moreover, if there is a component $G'$ of $G\setminus (X\cup C)$ with $w'(G')>\frac{1}{2}$, then no vertex of $G'$ is in more than $\frac{\log n}{32t(k-1)}$ of the sets $S_i$.

\sta{\label{claim 1/2 separator}$C$ is a $(w', 1/2)$-balanced separator in $G\setminus X$.}

Suppose not, and let  $G'$ be the unique
connected component of $G \setminus (X\cup C)$ with $w'(G') > \frac{1}{2}$.
Let $a_1,a_2,a_3\in V(G')$. We say that  $a_1a_2a_3$ is  {\em separated by $S_i$}
if for every component $D$ of $G' \setminus S_i$, $|D \cap \{a_1,a_2,a_3\}| \leq 1$.
Randomly choose three vertices $a_1,a_2, a_3$  of $G'$ using the normalized function of $w'$ on $G'$ as a probability distribution. Let $i \in \{1, \dots, \lceil \log n \rceil\}$.
Since $\frac{w'(D)}{w'(G')}\leq\frac{1}{8w'(G')} \leq \frac{1}{4}$ for every component $D$ of $G' \setminus S_i$, the probability that $a_1a_2a_3$ is separated by $S_i$ is
at least $\frac{3}{4}\cdot\frac{1}{2} = \frac{3}{8}$. Therefore, by the linearity of expectation, there exist $a_1,a_2,a_3 \in G'$  and $J \subseteq \{1, \dots, \lceil \log n \rceil\}$  with $|J| \geq \frac{3\log n}{8}$ such that $a_1a_2a_3$ is separated by $S_j$ for every $j \in J$.

For $1\le i,j\le 3$, $i\neq j$, let  $P_{ij}$ be a shortest path from $a_i$ to $a_j$ in $B$. Since no vertex of $G'$ belongs to more than $\frac{3\log n}{32t}$ of the sets $S_i$, and since $|J| \geq \frac{3\log n}{8}$, it follows that  $|P_{ij}| \geq 4t$ for every $i,j \in \{1,2,3\}$.

Let $Q_{ij}$ be the $t$ vertex subpath of $P_{ij}$ including $a_i$.
Since $Q_{ij}\setminus B = a_i$, it follows that $N_G[x]\cap Q_{ij}=\{a_i\}$.
Since $|P_{ij}| \geq  4t$ and by minimality of $P_{ij}$, it follows that $Q_{12}$, $Q_{23}$, and $Q_{31}$ are pairwise anticomplete.
Therefore, $x \cup  Q_{12}\cup Q_{23} \cup  Q_{31}$ is an $S_{t,t,t}$ with center $x$ in $G$ (see Figure~\ref{fig: claim 1/2 separator}), a contradiction.
This proves~\eqref{claim 1/2 separator}.
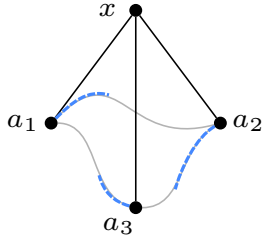
\begin{figure}[h]
    \centering
    \scalebox{1.5}{\input{./figures/slll.tikz}}
    \caption{Visualization of the $S_{t,t,t}$ obtained in the proof of \eqref{claim 1/2 separator}.}
    \label{fig: claim 1/2 separator}
\end{figure}
\\
\\

\sta{\label{claim weight reduction} There is no $(y,B')\in \mathcal{P}(G\sm C, X, \beta')$ such that $y=x$ and $B'\subseteq B$.}

Let $B'$ be a component of $B\setminus C$ and let $D$ be the component of $G\setminus (X\cup C)$ containing $B'$.
By \eqref{claim 1/2 separator}, $|D\cap I| \leq |I|/2$. Therefore, $I'=I\setminus \left(D \cup C\right)$ is an independent set of size at least $|I|-|I|/2- \eta \log^2n > \beta/4$, and $I'$ is anticomplete to $D$. Since $\left(I' \cup N_{G\setminus C}(B')\right) \cap N_{G\setminus C}(x) \subseteq N_{G}(B) \cap N_{G}(x) $ and by the maximality of $\beta$, we have that
\[\beta \geq \alpha\left(\left(I' \cup N_{G\setminus C}(B')\right)\cap N_{G\setminus C}(x)\right) > \beta/4  + \alpha\left(N_{G\setminus C}(B')\cap N_{G\setminus C}(x)\right)\,,\]
where the second inequality follows from the fact that the sets $I'\cap N_{G\setminus C}(x) = I'$ and $ N_{G\setminus C}(B')\cap  N_{G\setminus C}(x)\subseteq D$ are disjoint and anticomplete to each other.
Therefore, $\alpha\left(N_{G\setminus C}(B')\cap N_{G\setminus C}(x)\right) \leq \frac{3}{4}\beta$ and so $(x,B')$ is not a $(X,\frac{3}{4}\beta)$-problematic pair in $G\setminus C$.   This proves  \eqref{claim weight reduction}.
\\
\\
Note that for every pair $(y,D)\in \mathcal{P}(G\sm C,X,\beta')$, the pair $(y,D')$ where $D'$ is the unique component of $G$ containing $D$, belongs to $\mathcal{P}(G,X,\beta')$.
Therefore,
\begin{align*}
    W(G\sm C,X,\beta') &= \sum_{(y,D)\in \mathcal{P}(G\sm C,X,\beta')} w(D)\\
    &= \sum_{(y,D')\in \mathcal{P}(G,X,\beta')} \sum_{\substack{D\subseteq D'\\(y,D)\in \mathcal{P}(G\sm C,X,\beta')}} w(D)\\
    &=  \sum_{\substack{(y,D')\in \mathcal{P}(G,X,\beta')\\(y,D')\neq (x,B)}} \sum_{\substack{D\subseteq D'\\(y,D)\in \mathcal{P}(G\sm C,X,\beta')}} w(D)\,,
\end{align*}
where the last equality holds by \eqref{claim weight reduction}.\\
\\
Now
\[\sum_{\substack{(y,D')\in \mathcal{P}(G,X,\beta')\\(y,D')\neq (x,B)}} \sum_{\substack{D\subseteq D'\\(y,D)\in \mathcal{P}(G\sm C,X,\beta')}} w(D)
    \leq \sum_{\substack{(y,D')\in \mathcal{P}(G,X,\beta')\\(y,D')\neq (x,B)}} w(D')\,.\]
We deduce
\[ \sum_{\substack{(y,D')\in \mathcal{P}(G,X,\beta')\\(y,D')\neq (x,B)}} \sum_{\substack{D\subseteq D'\\(y,D)\in \mathcal{P}(G\sm C,X,\beta')}} w(D)
    \leq \sum_{(y,D')\in \mathcal{P}(G,X,\beta')} w(D') -\eps
    = W(G,X,\beta') - \eps\,,\]
as required.

\end{proof}

\begin{corollary}\label{corollary reducing max beta once}
Let $k,t\in \nat$ with $k\ge 2$, $\eps\in(0,1]$.
Let $\eta= c(16(k-1),3t+1,t,6t)$ be as in Theorem~\ref{thm : boosted construction result}.
Let $\mathcal{C}$ be a $k$-breakable class of graphs, and let
$G \in \mathcal{C}$ be  an $\set{S_{t,t,t},K_{t,t}}$-free graph on $n$ vertices, where $n\ge 2$.
Let  $w$ be a weight function on $G$. Let $X \subseteq G$ be such that $|X| < k$ and $N[X]$ is a $(w, \frac{1}{2})$-balanced separator for $G$. Let $\beta = \beta(G,X) > 4\eta\log^2n$.
Then, there exists $C^*\subseteq V(G\sm X)$ such that $\alpha(C^*)\leq \frac{k}{\eps}\eta\log^2n$ and $W(G\setminus C^*,X,\beta) = 0$.
\end{corollary}
\begin{proof}
    We recursively define $C_\ell\subseteq V(G\sm X)$ by applying Lemma \ref{lem : boosting one step} to $G\setminus \bigcup_{i< \ell}C_i$ with a weight function which is the appropriate restriction of $w$ and with $\beta'=\beta$.
    Let $\ell^*$ be the smallest $\ell$ such that $W(G\setminus\bigcup_{i\leq \ell}C_i,X,\beta) = 0$. Note that $\beta(G\setminus \bigcup_{i< \ell}C_i)\le \beta$ and, moreover, if $\beta(G\setminus \bigcup_{i< \ell}C_i)<\frac{3}{4}\beta$, then $W(G\setminus \bigcup_{i< \ell}C_i,X,\beta) = 0$, so the conditions of Lemma \ref{lem : boosting one step} are satisfied for $\ell\in \set{1,\dots,\ell^*-1}$.
    By the definition of the function $W$, we have that
    \[W(G,X,\beta)=\sum_{x\in X} \sum_{(x,B)\in \mathcal{P}(G,X,\beta)} w(B) \leq \sum_{x\in X} 1 = |X| \leq k\,.\]
    Therefore, $\ell^*\leq \frac{k}{\eps}$.
    Setting $C^*=\bigcup_{i\leq \ell^*} C_i$ completes the proof.
\end{proof}

\begin{corollary}\label{corollary reducing max beta}
    Under the same conditions as in Lemma \ref{lem : boosting one step}, there exist $C^*\subseteq V(G)\setminus X$ such that $\alpha(C^*) \leq \frac{k}{\eps}\eta\log^3n$ and the maximum value of $\gamma$ for which an $(X,\gamma)$-problematic pair exists in $G\setminus C^*$ is at most $4\eta \log^2n$.
\end{corollary}
\begin{proof}
    We recursively define $C_\ell\subseteq V(G\sm X)$ by applying Corollary \ref{corollary reducing max beta once} to $G\setminus \bigcup_{i< \ell}C_i$ with $\beta_i$ where $\beta_i = \beta(G\setminus \bigcup_{i< \ell}C_i,X)$ (so $\beta_1 =\beta(G,X)$).
    Since $W(G\setminus \bigcup_{j\leq i} C_j,X,\beta_i) = 0$, we have that $\mathcal{P}(G\setminus \bigcup_{j\leq i} C_j,X,\beta_i)=\mt$.
    Therefore, $\beta_{i+1}\leq \frac{3}{4}\beta_i$ and so \[n\geq \beta_1 \geq \frac{4}{3}\beta_2\geq  \dots \geq \left(\frac{4}{3}\right)^i\beta_i\,.\]
    Let $\ell^*$ be the smallest $\ell$ such that $\beta_\ell \leq  4\eta \log^2n$.  We have that $\ell^*\leq \log n$, so setting $C^*=\bigcup_{i\leq \ell^*} C_i$ completes the proof.
\end{proof}

We can now prove the main result of this subsection.
\begin{theorem}

\label{thm : boosting} For all positive integers $k,t$ with $k\ge 2$, there exist
an integer $c=c(k,t)$ with the following property.
Let $\mathcal{C}$ be a $k$-breakable class of graphs.
Let $G \in \mathcal{C}$ be an $\set{S_{t,t,t},K_{t,t}}$-free  graph on $n$ vertices, where $n\ge 2$.
Let $\eps\in(0,1]$ and let
$f=f_{k,t}:\nat\rightarrow\reel_{\geq0}$  be the function $f(n)=\frac{c}{\eps}\log^3n$.
Then $G$ is $(k,f,\eps)$-breakable.
\end{theorem}
\begin{proof}

Let $w$ be a weight function on $G$. Since $G$ is $k$-breakable, we can find $X \subseteq G$ such that $|X| < k$ and $N[X]$ is a $(w, \frac{1}{2})$-balanced separator in $G$.
Let $\beta = \beta(G,X)$ and let $\eta= c(16(k-1),3t+1,t,6t)$ be as in Theorem~\ref{thm : boosted construction result}. If $\beta > 4\eta\log^2n$, we apply Corollary \ref{corollary reducing max beta} to obtain a set $C\subseteq V(G)\setminus X$. Otherwise, let $C=\mt$. In both cases, $\mathcal{P}(G\setminus C, X, 4\eta\log^2(n)+1) = \mt$ which implies that none of the remaining big components of $G\setminus(C\cup N[X])$ has an independent set of size $4k\eta\log^2n$ in its neighborhood. Let $Z = \bigcup_{B\in \mathcal{B}(G\setminus(C\cup N[X]))}N(B)$ (see Figure~\ref{fig : boosting}).
Since there are at most $\frac{1}{\eps}$ big components,  $\alpha(Z)\leq \frac{1}{\eps}4k\eta\log^2n$.

\sta{\label{claim: boosting works} $(N[X],C\cup Z)$ is a $(w,\eps)$-boosted separator in $G$.}

Let $S$ be a component of $G\setminus (C\cup Z)$ with $w(S)$ maximum. If $w(S)\leq1/2$, there is nothing to show. Therefore, we may assume that $w(S)>1/2$. Since $N[X]$ is a $(w,\frac{1}{2})$-balanced separator of $G$, $S\cap N[X] \neq \mt$. Let $G^*$ be the union of all the components of $G\setminus(C\cup Z)$ intersecting $N_{G\setminus(C\cup Z)}[X]$. It follows that $S\subseteq G^*$. Since none of the components of $G^*\setminus N[X]$ is big, $N[X]$ is a $(w,\eps)$-balanced separator for $G^*$. Therefore,
\[S\sm N[X] = \left(G^*\sm (G^*\sm S)\right) \sm N[X] = \left(G^*\sm N[X]\right) \sm (G^*\sm S)\] has no component of weight greater than $\eps$, as removing vertices cannot introduce a big component. This proves \eqref{claim: boosting works}.
\begin{figure}[h]
    \centering
    \scalebox{1.3}{\input{./figures/boosting.tikz}}
    \caption{Visualization of $G\sm C$ in Theorem~\ref{thm : boosting}.}\label{fig : boosting}
\end{figure}
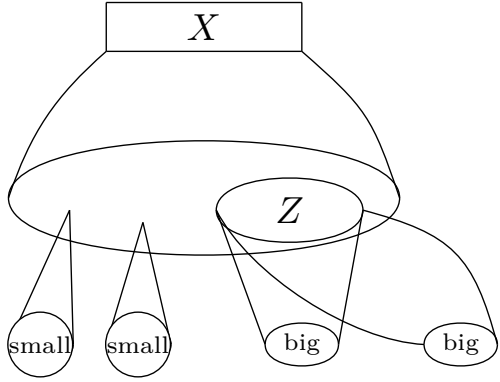

Since $\alpha(C\cup Z)\leq
 \frac{1}{\eps}k\eta\log^3(n)+ \frac{1}{\eps}4k\eta\log^2n$, setting $c = 4k\eta$ completes the proof.
\end{proof}

\subsection{Improving the Disjointedness of the Separators}
\label{sec: disjoint}
Our next goal is to combine Theorems~\ref{thm:domsep}, \ref{thm : boosted construction result}, and \ref{thm : boosting}  to obtain a large set of boosted separators with pairwise anticomplete cores.

\begin{lemma}\label{lemma disjoint cores}
Let $k,N\in \nat$ and $R\in\reel_{\geq0}$. Let $G$ be a graph. Let $Y_1, \dots, Y_{\lceil kRN\rceil}$ be subsets of $V(G)$ of size less than $k$.
Write $S_i=N[Y_i]$.
Assume that for all $j\in \{1,\ldots, \lceil kRN\rceil\}$, no vertex of $Y_j$ belongs to more than $R$ of the sets $S_i$ with $i<j$.
Then there exists $I \subseteq \{1, \dots, \lceil kRN\rceil\}$
with $|I|= N$ such that for every $i \neq j \in I$, $Y_i$ is anticomplete to $Y_j$.
\end{lemma}
\begin{proof}
Let $H$ be the graph with vertex set $\{Y_1,\ldots, Y_{\lceil kRN\rceil}\}$, with $Y_i$ adjacent to $Y_j$ for $i\neq j$ if and only if $Y_i$ and $Y_j$ are not anticomplete to each other.
Since $|Y_j|<k$ for all $j$, the graph $H$ is $\lfloor kR \rfloor$-degenerate, and therefore has a stable set of size $N$, as required.
\end{proof}

\begin{lemma}\label{lemma: no vertex in t separator}
Let $t,N>1$ be integers, and let $G$ be a $K_{t,t}$-free graph. Let  $w$ be a weight function on $G$.
Let $Y_1, \dots, Y_N$ be pairwise anticomplete subsets of size at most $k$ of $V(G)$,
and for every $i \in \{1, \dots, N\}$ write $S_i=N[Y_i]$.
Let $Z= \set{v \in V(G) \colon \abs{\set{i \colon v\in S_i}}\geq t}$.
Then $\alpha(Z) \leq \binom{N}{t} tk^t$.

\end{lemma}
\begin{proof}
We start with the following:

\sta{$Z \cap Y_i=\emptyset$ for every $i \in \{1, \dots, N\}$. \label{ZY}}

Suppose there is a vertex $z \in Z \cap Y_1$.  Since $t>1$, we may assume that $z \in S_2$.
But  since $z \in Y_1$ and since  the sets $Y_1, \dots, Y_N$ are pairwise
anticomplete, it follows that $z$ is anticomplete to $Y_2$, a contradiction.
This proves~\eqref{ZY}.
\\
\\
Now let  $I$ be a largest stable set in $Z$.
We may assume that $|I|>  \binom{N}{t} t k^t$.
Then $I$ contains a subset $I_1$ of size $tk^t$ such that every $v \in I_1$ is in $S_1,\dots,S_t$ (renumbering $S_1,\dots,S_N$ if necessary). Since by \eqref{ZY} $Z\cap Y_i = \mt$ for every $i$, it follows that every vertex in $I_1$ has a neighbor in each of $Y_1,\dots,Y_t$. Now there exists a subset $I_2$ of $I_1$ of size $t$ such that for every $a \in \{1, \dots, t\}$ there exists $y_a \in Y_a$, for which every $v \in I_2$ is adjacent to $y_a$.
But now $I_2 \cup \{y_1,\dots,y_t\}$ is a $K_{t,t}$ in $G$, a contradiction.
\end{proof}

We can now prove the main result of this subsection.

\begin{theorem} \label{thm:disjoint separators}
Let $t$ be a positive integer and let $\eps\in(0,1]$.
Then, there exists an integer $c=c(t,\eps)$ with the following properties.
Let $d = d(t)$ be as in Theorem~\ref{thm:domsep}.
Let $G \in \mathcal{M}_t$ with $|G|=n \ge 2$ and let $w$ be a weight function on $G$.
Then, there is a set $X \subseteq V(G)$ with  $\alpha(X) \leq c \log^4 n$, such that, denoting by $D$ a component of $G \setminus X$ with $w(D)$ maximum, either
\begin{itemize}
\item $w(D) \leq \frac{1}{2}$ (and therefore $X$ is a  $(w, \frac{1}{2})$-balanced separator of $G$), or
\item $w(D) > \frac{1}{2}$ and
there exist  pairwise anticomplete subsets  $Y_1, \dots Y_{8t^2d}$
of $G$ such that
\begin{itemize}
\item for every $i$,  $|Y_i| <d$, and
\item  $N[Y_i]\cap D$ is a $(w, \epsilon)$-balanced separator of $D$, and
\item no vertex of $D$ is in more than $t$ of the sets $N[Y_i]$.
\end{itemize}

\end{itemize}
\end{theorem}

\begin{proof}
Let $N= 8t^2d$ and $\lambda = 3N$.
By Theorem \ref{thm:domsep}, we know that $\mathcal{M}_t^*$ is $d$-breakable.
We may assume that $d\ge 2$.
It follows that Theorem \ref{thm : boosting} applies and there exists  $c_1=c(d,t)$ such that $G$ is $(d,f,\eps)$-breakable where $f(n)=\frac{c_1}{\eps}\log^3n$.
We can therefore apply Theorem \ref{thm : boosted construction result} to get $c_2=c(d,3t+1,t,\lambda)$ and $C,Y_1,\dots,Y_{\lceil \log n \rceil} \subseteq V(G)$ such that (writing $S_i=N[Y_i]$)
\begin{itemize}
    \item $\alpha(C) \leq c_2\log^4n$;
    \item for every $i$, $|Y_i|<d$;
    \item for every $i$,  $(S_i,C)$ is a $(w,\eps)$-boosted separator of $G$;
    \item if there is a component $G'$ of $G\setminus C$ with $w(G')>\frac{1}{2}$, then no vertex of $G'$ is contained in more than $\frac{3\log n}{d\lambda}$ of the sets $S_1, \dots, S_{\lceil \log n \rceil}$;
    \item for every $j\in \{1,\ldots, \lceil \log n \rceil\}$, no vertex of $Y_j$ is contained in more than $\frac{3\log n}{d\lambda}$ of the sets $S_i$ with $i<j$.
    \end{itemize}

Let $G'$ be the component of $G\setminus C$ with $w(G')$ maximum.
Let $c = c_2+\binom{N}{t}td^t$.
We may assume that $w(G')>1/2$ as otherwise, we are done.
Applying Lemma \ref{lemma disjoint cores} to $G'$ with $k = d$ and $R=\frac{3\log n}{d\lambda}$
we obtain  $I \subseteq \{1, \dots, \lceil kRN\rceil\}$
with $|I|= N$ such that for every $i \neq j \in I$, $Y_i$ is anticomplete to
$Y_j$.  Renumbering if necessary, we may assume that $I=\{1, \dots, N\}$.
Applying Lemma \ref{lemma: no vertex in t separator} to the sets $Y_1, \dots, Y_N$ gives a set $Z$ with $\alpha(Z)\leq \binom{N}{t} td^t = c_3$  such that every vertex of $G \setminus Z$ is in fewer than $t$ of the sets $S_1, \dots, S_N$.
Let $X=C \cup Z$. Then $\alpha(X) \leq c_3+c_2\log^4n$.
Let $D$ be the component of $G\setminus X$ with $w(D)$ maximum.  If
$w(D)\leq 1/2$, then the first alternative in the statement of the theorem holds,
 and we are done. Therefore, we may assume that $w(D)> 1/2$.
 But now the second alternative in the statement of the theorem holds.
\end{proof}

\section{Bringing it all together} \label{sec:claw}

The final key ingredient in our proof of Theorem \ref{thm claw} is the following lemma.
By a \textit{caterpillar}, we mean a tree $T$ with maximum degree of three such that there exists a path that contains all the vertices of degree three. For a graph $H$, we denote by $\mathcal{Z}(H)$ the set of vertices whose neighborhood is a clique.

\begin{lemma}[Theorem 5.2 of \cite{tw15}]
\label{lem:conntw15}
For every integer $h \geq 1$, there exists an integer $\mu = \mu(h) \geq 1$ with the following property. Let $G$ be a connected graph. Let $Y \subseteq G$ such that $|Y| \geq \mu$, $G\setminus Y$ is connected and every vertex of $Y$ has a neighbor in $G \setminus Y$. Then there is a set $Y' \subseteq Y$ with $|Y'| = h$ and an induced subgraph $H$ of $G \setminus Y$ for which one of the following holds.
\begin{itemize}
\item $H$ is a path and every vertex of $Y'$ has a neighbor in $H$; or
\item $H$ is a caterpillar, or the line graph of a caterpillar, or a subdivided star, or the line graph of a subdivided star. Moreover, every vertex
of $Y'$ has a unique neighbor in $H$ and $H \cap N(Y') = \mathcal{Z}(H)$.
\end{itemize}

\end{lemma}

We prove a slightly modified version of Lemma~\ref{lem:conntw15} that will be more convenient for our use.

\begin{lemma}\label{lemma : connectifier}
For every integer $h \geq 1$, there exists an integer $\mu = \mu(h) \geq 1$ with the following property. Let $G$ be a connected graph. Let $S \subseteq G$  be a stable set such that $|S| \geq \mu$. Then there is an induced subgraph $H$ of $G$ for which one of the following conditions holds:
\begin{itemize}
\item $H$ is a path and $|H \cap S| = h$.
\item $H$ is a caterpillar, or the line graph of a caterpillar, or a subdivided star, or the line graph of a subdivided star with $|H \cap S| = h$ and $H \cap S = \mathcal{Z}(H)$.
\end{itemize}

\end{lemma}
\begin{proof}
Let $\mu=\mu(h^2)$ be as in Lemma~\ref{lem:conntw15}.
Define $G'$ to be the graph obtained from $G$  by adding,
for each $v \in S$, a new  vertex $u_v$ whose unique neighbor in $G'$ is  $v$. Let $Y=\{u_v\colon v \in S\}$ (so $Y=G' \setminus G$).
Then  $|Y| =|S| \geq \mu$, every vertex in $Y$ has a neighbor in $G' - Y$, and $G' - Y = G$ is connected.

Applying Lemma~\ref{lem:conntw15} to $G'$ and $Y$, we obtain a set $Y' \subseteq Y$ with $|Y'| = h^2$ and an induced subgraph $H$ of $G' - Y = G$ for which one of the following holds:
\begin{itemize}
\item $H$ is a path and every vertex of $Y'$ has a neighbor in $H$.
Since every $v \in S$ has a unique neighbor in $Y$, and since $Y$
is anticomplete to $G \setminus S$, it follows that
$|H \cap S|  \geq |Y'| \geq h$. Truncating $H$, we obtain the required conclusion.
\item $H$ is a caterpillar, or the line graph of a caterpillar, or a subdivided star, or the line graph of a subdivided star. Moreover, every vertex
of $Y'$ has a unique neighbor in $H$ and $H \cap N(Y') = \mathcal{Z}(H)$.
Again, since $|N(y)|=1$ for every $y \in Y$, it is
easy to see that $H \setminus Y$ is a caterpillar, or the line graph of a caterpillar, or a subdivided star, or the line graph of a subdivided star.
We may assume that for every path $P$ in $H$, $|P \cap S| < h$, for otherwise the theorem holds.
Since $S$ is stable,  the observation in the previous sentence implies that
there exists an induced subgraph $H'$ of $H$, and $Y'' \subseteq Y'$ with $|Y''|=h$
such that $H'$ is a caterpillar, or the line graph of a caterpillar, or a subdivided star, or the line graph of a subdivided star,
$N(Y'') = \mathcal{Z}(H')$, and
every vertex of $N(Y) \cap H' = H' \cap S$ belongs to $\mathcal{Z}(H')$.
It follows that $H'=\mathcal{Z}(H')$, as required.\qedhere
\end{itemize}
\end{proof}

We can now prove Theorem \ref{thm claw}, which we restate.

\claw*

\begin{proof}
 Let $G$ be an $n$-vertex graph in $\mathcal{M}_t^*$ and let $d=d(t)$ be the constant given by Theorem \ref{thm:domsep}.
Let $N= 8t^2d,h=10Nd$
and $\eps = \frac{1}{2\mu^2}$ where $\mu = \mu(h)$ as in Lemma \ref{lemma : connectifier}.
Let $c=c(t, \eps)$ be as in Theorem~\ref{thm:disjoint separators}.
By Theorem~\ref{thm:disjoint separators} we may assume that
there exists $Z \subseteq V(G)$ with $\alpha(Z) \leq c \log^4 n$ such that,
denoting by $G''$ the component of $G \setminus Z$ with $w(G'')$ maximum, $w(G'')> \frac{1}{2}$ and   there exist  pairwise anticomplete subsets  $Y_1, \dots Y_{8t^2d}$
of $G$ such that
\begin{itemize}
\item for every $i$,  $|Y_i| <d$, and
\item  $N[Y_i]\cap G''$ is a $(w, \epsilon)$-balanced separator of $G''$, and
\item no vertex of $G''$ is in more than $t$ of the sets $N[Y_i]$.
\end{itemize}
We now randomly choose  $\mu$ vertices of $G''$ using the normalized function of $w$ on $G''$ as a probability distribution.
For every $i$, write $S_i=N[Y_i]$.
For every $i$ and for every component $D$ of $G''\setminus S_i$,
$\frac {w(D)}{w(G'')}  \leq 2 \eps $, and so, applying the union bound, the probability that no two of the vertices we chose  are in the same component of $G''\setminus S_i$ is at least
\[1-2\eps\sum_{j=1}^{\mu(h)}j \geq 1-\eps\mu(h)^2 = 1/2\,.\] By the linearity of expectation, there exist $S\subseteq G''$ with $|S|=\mu$
and a set $I \subseteq \{1, \dots, N\}$  with $|I| =\frac{N}{2}$
such that for every $i \in I$ and every
component $D$ of $G'' \setminus S_i$, $|S \cap D| \leq 1$.

\sta{$S$ is a stable set. \label{stableS}}

Suppose $s,t \in S$ are adjacent. Since  for every component $D$ of $G'' \setminus S_i$,
$|S \cap D| \leq 1$, it follows that for every $i \in \{1, \dots, N\}$,
$|\{s,t\} \cap N[Y_i]| \geq 1$. Since no vertex of $G$ is in more than $t$ of the sets
$N[Y_i]$, it follows that $N \leq 2t$, a contradiction. This proves \eqref{stableS}.
\\
\\
In view of \eqref{stableS} we now  apply Lemma \ref{lemma : connectifier} to $G''$ and $S$  to  obtain an induced subgraph $H$ of $G''$  where $X = \mathcal{Z}(H) \cap S = \{x_1,\dots,x_h\}$  and
$H$ is either a path, a caterpillar, the line graph of a caterpillar, a subdivided star, or the line graph of a subdivided star.  We will get a contradiction in each of the cases.
We start with the following:

\sta{ Let $1 \leq i < j \leq h$. Then for every path $P$  from $x_i$ to $x_j$
in $G''$ and for every $\ell \in I$, $S_\ell \cap P \neq \emptyset$.
Consequently, $|P| \geq 4t+2$. \label{farapart}}

Let $P$ be a path from $x_i$ to $x_j$ in $G''$.
Suppose $P \cap S_\ell= \emptyset $ for
some $\ell \in I$.
Then, there is a component $D$ of $G'' \setminus S_\ell$ such that $x_i,x_j \in D$,
a contradiction.
This proves that $P \cap S_\ell \neq \emptyset$ for every $\ell \in I$.
Since $N/2 \geq  4t^2$, and since no vertex of $G''$ is in more than $t$ of the sets $N[Y_i]$,
it follows that  $|P \setminus \{x_i,x_j\} | \ge  4t$, as required. This completes the proof of
\eqref{farapart}.
\\
\\
\sta{\label{claim no path caterpillar or L(caterpillar)}$H$ is not a path, a caterpillar, or the line graph of a caterpillar.}
Assume that $H$ is a path, a caterpillar, or the line graph of a caterpillar.
Without loss in generality, assume that  $x_1,\ldots,x_h$ appear in the natural order given by $H$ (see Figure~\ref{fig: path, caterpillar, L(caterpillar)}).
\begin{figure}[h]
    \centering
    \scalebox{0.65}{\input{./figures/path_caterpillar_line_graph_of_caterpillar.tikz}}
    \caption{A path, a caterpillar, and the line graph of a caterpillar with $X$ ordered.}
    \label{fig: path, caterpillar, L(caterpillar)}
\end{figure}
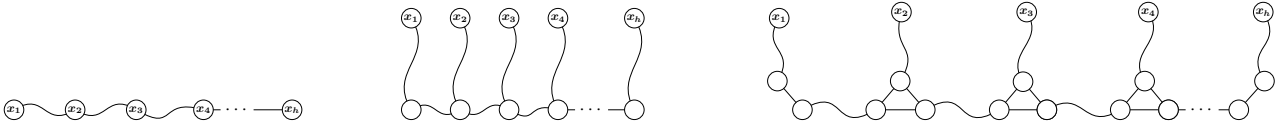

For each $i = 1,\ldots,h/2$, let $P_i$ be a path in $H$ from $x_{2i-1}$ to $x_{2i}$.

Let $i \in \{1, \dots, \frac{h}{2}\}$.
By \eqref{farapart}, there is a minimal subpath  $P_i'$ of $P_i$  containing $x_{2i-1}$ and such  that $P_i' \cap S_j \neq \emptyset$ for every $j \in I$.
Since $N/2> t^2$, and since no vertex of $G$ is in more than $t$ of the sets $N[Y_j]$, it follows that  $|P'_i| >  t$.
Let $y_i$ be the end of $P_i'$  different from $x_{2i-1}$.
By the minimality of $P_i'$, there exists $j \in I$ such that $y_i \in S_j$,
and no other vertex of $P_i'$ belongs to $S_j$; write $j=j(i)$.
Since $\frac{h}{2} \geq 50Nd > 3d|I|$, there exists $j \in I$ such that
$j(i)=j$ for at least $3d$ values  of $i$. Since $|Y_j|<d$ for every $j \in I$,
there is $y \in Y_j$ and distinct $i_1,i_2,i_3 \in \{1, \dots, \frac{h}{2}\}$ such that $y$ is complete to $\{y_{i_1}, y_{i_2}, y_{i_3}\}$.
Recall  that $y$ is anticomplete to $(P_{i_1}' \setminus y_{i_1})  \cup (P_{i_2}' \setminus y_{i_2}) \cup  (P_{i_3}' \setminus y_{i_3})$. But now there is an $S_{t,t,t}$ in $G$
with center $y$ and paths contained in $P_{i_1}'$, $P_{i_2}'$, and $P_{i_3}'$, a contradiction.
This proves~\eqref{claim no path caterpillar or L(caterpillar)}.
\\
\\
\sta{\label{claim: not star} $H$ is not a subdivided star.}
Suppose that $H$ is a subdivided star with $|H \cap X| = h$ and $H \cap X = \mathcal{Z}(H)$.
Let $z$ be the unique vertex of $H$ with degree at least $3$.
Since $h \geq 4$,
there exist paths $P_1, P_2,P_3, P_4$ of $H$ where $P_i$ is from
$z$ to $x_i$.   By \eqref{farapart},   for at least three values of $i$,
say $i \in \{1,2,3\}$, $|P_i|> t$. But now $P_1 \cup P_2 \cup P_3$
contains an $S_{t,t,t}$ with center $z$ in $G$ (see Figure~\ref{fig: not star}), a contradiction.
This proves \eqref{claim: not star}.

\begin{figure}[h]
    \centering
    \scalebox{1.3}{\input{./figures/claw_case.tikz}}
    \caption{Visualization of the $S_{t,t,t}$ obtained to prove \eqref{claim: not star}.}\label{fig: not star}
\end{figure}
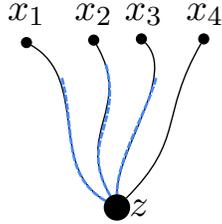

By \eqref{claim no path caterpillar or L(caterpillar)} and \eqref{claim: not star},
it follows that $H$ is the line graph of a subdivided star. Then
$H$ consists of a clique $K=\{k_1, \dots, k_h\}$ and paths
$P_1,  \dots, P_h$, where $P_i$ is from $k_i$ to $x_i$.
For every $i \in \{1, \dots, h\}$ let $I(i) \subseteq I$ be the set of all $j\in I$ such that $S_j \cap (P_i \setminus k_i)  \neq \emptyset$.

\sta{$|I(i)| < \frac{N}{8}$ for at most one value of $i$. \label{Ii}}

Suppose $|I(1)|, |I(2)| < \frac{N}{8}$. Since each of $k_1$, $k_2$ belongs to at most
$t$ of the sets $S_i$, it follows that $P_1 \cup P_2$ meets $S_i$ for at most
$2 \frac{N}{8} + 2t < \frac{N}{2} = |I|$ values of $i$, contrary to
\eqref{farapart}. This proves~\eqref{Ii}.
\\
\\
By renumbering if necessary, we can assume that  $|I(i)|  \geq \frac{N}{8}$
for every $i \in \{1, \dots, \frac{h}{2}\}$.
Let $i \in \{1, \dots, \frac{h}{2}\}  $.
Since $|I(i)| \geq \frac{N}{8}$, there is a minimal subpath  $P_i'$ of $P_i$  containing $x_i$ and such  that $P_i' \cap S_j \neq \emptyset$ for every $j \in I(i)$.
Since $N/8 >  t^2$, and since no vertex of $G$ is in more than $t$ of the sets $N[Y_i]$,
it follows that  $|P'_i|  > t$.
Let $y_i$ be the end of $P_i'$  different from $x_i$.
By the minimality of $P_i'$, there exists $j \in I$ such that $y_i \in S_j$,
and no other vertex of $P_i'$ belongs to $S_j$; write $j=j(i)$.
Since $\frac{h}{2} \geq 50Nd > 3d|I|$, there exists $j \in I$ such that
$j(i)=j$ for at least $3d$ values   of $i$. Since $|Y_j|<d$ for every $j \in I$,
there is $y \in Y_j$ and distinct $i_1,i_2,i_3 \in \{1, \dots, \frac{h}{2}\}$  such that $y$ is complete to $\{y_{i_1}, y_{i_2}, y_{i_3}\}$.
Recall that  $y$ is anticomplete to $(P_{i_1}' \setminus y_{i_1})  \cup (P_{i_2}' \setminus y_{i_2}) \cup  (P_{i_3}' \setminus y_{i_3})$. But now there is an $S_{t,t,t}$ in $G$
with center $y$ and paths contained in $P_{i_1}'$, $P_{i_2}'$, $P_{i_3}'$, a contradiction.
\end{proof}

\section{Acknowledgment}

We are grateful to Pawe\l{} Rz\k{a}\.zewski for allowing us to use Figure~\ref{fig:ex esd}.
Part of this research was conducted while the fourth author visited Princeton University.
The final phase of the writing of this paper took place at the Dagstuhl Seminar 25041, ``Solving Problems on Graphs: From Structure to Algorithms'', Jan 19--24, 2025.
We are grateful to both institutions for their hospitality.

\end{document}

%% file: graph.tikzstyles
\usetikzlibrary {decorations.pathreplacing}

\tikzstyle{box}=[shape=rectangle, text height=1.5ex, text depth=0.25ex, yshift=0.5mm, fill=white, draw=black, minimum height=5mm, yshift=-0.5mm, minimum width=5mm, font={\small}]
\tikzstyle{gate}=[shape=rectangle, text height=1.5ex, text depth=0.25ex, yshift=0.5mm, fill=white, draw=black, minimum height=5mm, yshift=-0.5mm, minimum width=5mm, font={\small}, tikzit category=circuit]
\tikzstyle{big gate}=[shape=rectangle, text height=1.5ex, text depth=0.25ex, yshift=0.5mm, fill=white, draw=black, minimum height=10mm, yshift=-0.5mm, minimum width=5mm, font={\small}, tikzit category=circuit]
\tikzstyle{Z dot}=[inner sep=0mm, minimum size=2mm, shape=circle, draw=black, fill={rgb,255: red,221; green,255; blue,221}, tikzit category=zx]
\tikzstyle{Z phase dot}=[minimum size=5mm, font={\footnotesize\boldmath}, shape=rectangle, rounded corners=2mm, inner sep=0.2mm, outer sep=-2mm, scale=0.8, tikzit shape=circle, draw=black, fill=white, tikzit draw=blue, tikzit category=zx]
\tikzstyle{X dot}=[Z dot, shape=circle, draw=black, fill={rgb,255: red,255; green,136; blue,136}, tikzit category=zx]
\tikzstyle{X phase dot}=[Z phase dot, tikzit shape=circle, tikzit draw=blue, fill={rgb,255: red,255; green,136; blue,136}, font={\footnotesize\boldmath}, tikzit category=zx]
\tikzstyle{hadamard}=[fill=yellow, draw=black, shape=rectangle, inner sep=0.6mm, minimum height=1.5mm, minimum width=1.5mm, tikzit category=zx]
\tikzstyle{paulibox}=[fill={rgb,255: red,221; green,221; blue,255}, draw=black, shape=rectangle, inner sep=0.6mm, minimum height=5mm, minimum width=5mm, font={\footnotesize}, text height=1.5ex, text depth=0.25ex, tikzit category=zx]
\tikzstyle{vertex}=[inner sep=0mm, minimum size=1mm, shape=circle, draw=black, fill=black, tikzit category=misc]
\tikzstyle{vertex set}=[inner sep=0mm, minimum size=2mm, shape=circle, draw=black, fill=white, font={\footnotesize\boldmath}, tikzit category=misc]
\tikzstyle{small black dot}=[fill=black, draw=black, shape=circle, inner sep=0pt, minimum width=1.2mm, tikzit category=circuit]
\tikzstyle{cnot ctrl}=[fill=black, draw=black, shape=circle, inner sep=0pt, minimum width=1.2mm, tikzit category=circuit]
\tikzstyle{cnot targ}=[fill=white, draw=white, shape=circle, tikzit category=circuit, label={center:$\oplus$}, inner sep=0pt, minimum width=2.1mm, tikzit fill={rgb,255: red,102; green,204; blue,255}, tikzit draw=black]
\tikzstyle{ket}=[fill=white, draw=black, shape=regular polygon, regular polygon sides=3, regular polygon rotate=-30, scale=0.7, inner sep=1pt, tikzit category=circuit, tikzit shape=rectangle, tikzit fill=green]
\tikzstyle{bra}=[fill=white, draw=black, shape=regular polygon, regular polygon sides=3, regular polygon rotate=30, scale=0.7, inner sep=1pt, tikzit category=circuit, tikzit shape=rectangle, tikzit fill=red]
\tikzstyle{scalar}=[shape=rectangle, text height=1.5ex, text depth=0.25ex, yshift=0.5mm, fill=white, draw=black, minimum height=5mm, yshift=-0.5mm, minimum width=5mm, font={\small}]
\tikzstyle{clabel}=[fill=white, draw=none, shape=rectangle, tikzit fill={rgb,255: red,56; green,255; blue,242}, font={\footnotesize}, inner sep=1pt, tikzit category=labels]
\tikzstyle{empty diagram}=[draw={gray!40!white}, dashed, shape=rectangle, minimum width=1cm, minimum height=1cm, tikzit category=misc]
\tikzstyle{white dot}=[Z dot]
\tikzstyle{gray dot}=[X dot]
\tikzstyle{white phase dot}=[Z phase dot]
\tikzstyle{gray phase dot}=[X phase dot]
\tikzstyle{small hadamard}=[hadamard]

\tikzstyle{simple}=[-]
\tikzstyle{hadamard edge}=[-, dashed, dash pattern=on 2pt off 0.5pt, thick, draw={rgb,255: red,68; green,136; blue,255}]
\tikzstyle{non edge}=[-, dashed, dash pattern=on 2pt off 1pt, draw=black]
\tikzstyle{box edge}=[-, dashed, dash pattern=on 2pt off 0.5pt, thick, draw={rgb,255: red,203; green,192; blue,225}]
\tikzstyle{brace edge}=[-, tikzit draw=blue, decorate, decoration={brace,amplitude=1mm,raise=-1mm}]
\tikzstyle{diredge}=[->]
\tikzstyle{double edge}=[-, double, shorten <=-1mm, shorten >=-1mm, double distance=2pt]
\tikzstyle{gray edge}=[-, {gray!60!white}]
\tikzstyle{pointer edge}=[->, very thick, gray]
\tikzstyle{boldedge}=[-, line width=1.6pt, shorten <=-0.17mm, shorten >=-0.17mm]

%% file: figures/pattern.tikz
\begin{tikzpicture}
	\begin{pgfonlayer}{nodelayer}
		\node [style=vertex] (0) at (-2, 1.5) {};
		\node [style=vertex] (1) at (-2, -1.5) {};
		\node [style=vertex] (2) at (1, 0) {};
		\node [style=vertex] (3) at (4, 0) {};
		\node [style=vertex] (4) at (6, 2) {};
		\node [style=vertex] (5) at (6, -2) {};
		\node [style=vertex] (6) at (-3.25, 3) {};
		\node [style=vertex] (7) at (-3.25, -3) {};
		\node [style=none] (8) at (-4, 3.25) {$a$};
		\node [style=none] (9) at (-2.75, 1.5) {$c$};
		\node [style=none] (10) at (-2.5, -1.25) {$d$};
		\node [style=none] (11) at (-3.75, -3) {$b$};
		\node [style=none] (12) at (1, 0.75) {$e$};
		\node [style=none] (13) at (5, 0) {$f$};
		\node [style=none] (14) at (6.75, 2.25) {$g$};
		\node [style=none] (15) at (6.75, -1.75) {$h$};
	\end{pgfonlayer}
	\begin{pgfonlayer}{edgelayer}
		\draw (6) to (0);
		\draw (1) to (0);
		\draw (0) to (2);
		\draw (2) to (3);
		\draw (3) to (4);
		\draw (3) to (5);
		\draw (2) to (1);
		\draw (1) to (7);
	\end{pgfonlayer}
\end{tikzpicture}

%% file: figures/pattern_prime.tikz
\begin{tikzpicture}
	\begin{pgfonlayer}{nodelayer}
		\node [style=vertex] (0) at (-3, 2) {};
		\node [style=vertex] (1) at (-3, -2) {};
		\node [style=vertex] (2) at (0.5, 0) {};
		\node [style=vertex] (3) at (4, 0) {};
		\node [style=vertex] (4) at (6.5, 2.25) {};
		\node [style=vertex] (5) at (6.5, -2.25) {};
		\node [style=vertex] (6) at (-4.5, 4) {};
		\node [style=vertex] (7) at (-4.5, -4) {};
		\node [style=none] (8) at (-5, 4.175) {$a$};
		\node [style=none] (9) at (-3.525, 2) {$c$};
		\node [style=none] (10) at (-3.5, -1.75) {$d$};
		\node [style=none] (11) at (-5, -4) {$b$};
		\node [style=none] (12) at (0.5, 0.5) {$e$};
		\node [style=none] (13) at (4.8, 0) {$f$};
		\node [style=none] (14) at (7.025, 2.3) {$g$};
		\node [style=none] (15) at (7.075, -2.2) {$h$};
		\node [style=vertex] (16) at (-1.275, 0.975) {};
		\node [style=vertex] (17) at (-1.375, -1.1) {};
		\node [style=vertex] (18) at (-3, 0) {};
		\node [style=vertex] (19) at (-3.75, 3.025) {};
		\node [style=vertex] (20) at (-3.75, -3) {};
		\node [style=vertex] (21) at (2.25, 0) {};
		\node [style=vertex] (22) at (5.25, 1.15) {};
		\node [style=vertex] (23) at (5.275, -1.175) {};
		\node [style=none] (24) at (-1, 1.5) {$v_{ce}$};
		\node [style=none] (26) at (-3.75, 0) {$v_{cd}$};
		\node [style=none] (27) at (-0.975, -1.575) {$v_{de}$};
		\node [style=none] (28) at (6.25, 1) {$v_{fg}$};
		\node [style=none] (29) at (6, -1) {$v_{fh}$};
		\node [style=none] (30) at (2.25, 0.475) {$v_{ef}$};
		\node [style=none] (31) at (-4.3, -2.75) {$v_{bd}$};
		\node [style=none] (32) at (-4.5, 3) {$v_{ac}$};
		\node [style=vertex] (33) at (-3, 4) {};
		\node [style=vertex] (34) at (-3, -4) {};
		\node [style=vertex] (35) at (-2, -2.75) {};
		\node [style=vertex] (36) at (-1.5, 2.5) {};
		\node [style=vertex] (37) at (0.5, -1.25) {};
		\node [style=vertex] (38) at (4, -1.25) {};
		\node [style=vertex] (39) at (5.75, 3.25) {};
		\node [style=vertex] (40) at (5.75, -3) {};
		\node [style=none] (41) at (-2.25, 4) {$v_{a}$};
		\node [style=none] (42) at (6.25, 3.5) {$v_{g}$};
		\node [style=none] (43) at (-2.25, -4) {$v_{b}$};
		\node [style=none] (44) at (-1.25, -2.75) {$v_{d}$};
		\node [style=none] (45) at (0.5, -1.9) {$v_{e}$};
		\node [style=none] (46) at (4, -1.9) {$v_{f}$};
		\node [style=none] (47) at (5.75, -3.5) {$v_{h}$};
		\node [style=none] (48) at (-0.8, 2.525) {$v_{c}$};
		\node [style=none] (49) at (-1.35, 0.35) {$v_{cde}$};
		\node [style=vertex] (50) at (-1.75, 0) {};
	\end{pgfonlayer}
	\begin{pgfonlayer}{edgelayer}
		\draw (0) to (16);
		\draw (16) to (2);
		\draw (2) to (17);
		\draw (17) to (1);
		\draw (1) to (18);
		\draw (18) to (0);
		\draw (6) to (19);
		\draw (19) to (0);
		\draw (1) to (20);
		\draw (20) to (7);
		\draw (3) to (22);
		\draw (22) to (4);
		\draw (3) to (23);
		\draw (23) to (5);
		\draw (21) to (3);
		\draw (21) to (2);
		\draw (5) to (40);
		\draw (3) to (38);
		\draw (2) to (37);
		\draw (4) to (39);
		\draw (6) to (33);
		\draw (0) to (36);
		\draw (1) to (35);
		\draw (7) to (34);
		\draw (50) to (2);
		\draw (50) to (0);
		\draw (50) to (1);
	\end{pgfonlayer}
\end{tikzpicture}

%% file: figures/d_separable_non_consecutive.tikz
\begin{tikzpicture}
	\begin{pgfonlayer}{nodelayer}
		\node [style=none] (0) at (-10, 0) {};
		\node [style=none] (1) at (-7, 0) {};
		\node [style=none] (2) at (-3, 0) {};
		\node [style=none] (3) at (0, 0) {};
		\node [style=none] (4) at (4, 0) {};
		\node [style=none] (5) at (7, 0) {};
		\node [style=none] (6) at (10.025, 0) {};
		\node [style=none] (7) at (-13, 0) {};
		\node [style=none] (8) at (-10, 0.325) {};
		\node [style=none] (9) at (-7, 0.325) {};
		\node [style=none] (10) at (-3, 0.325) {};
		\node [style=none] (11) at (0, 0.325) {};
		\node [style=none] (12) at (4, 0.325) {};
		\node [style=none] (13) at (7, 0.325) {};
		\node [style=none] (14) at (-8.5, 0.825) {$P_i$};
		\node [style=none] (15) at (-1.425, 0.825) {$P_j$};
		\node [style=none] (16) at (5.5, 0.825) {$P_\ell$};
		\node [style=vertex] (17) at (-1.5, -3) {};
		\node [style=none] (18) at (-9, 0) {};
		\node [style=none] (19) at (-7.75, 0) {};
		\node [style=none] (20) at (4.75, 0) {};
		\node [style=vertex] (21) at (6, 0) {};
		\node [style=none] (22) at (-1.5, -3.5) {$v$};
		\node [style=none] (23) at (6, -0.5) {$p_m$};
	\end{pgfonlayer}
	\begin{pgfonlayer}{edgelayer}
		\draw (7.center) to (0.center);
		\draw (0.center) to (1.center);
		\draw (2.center) to (1.center);
		\draw (2.center) to (3.center);
		\draw (3.center) to (4.center);
		\draw (4.center) to (5.center);
		\draw (5.center) to (6.center);
		\draw [style=brace edge] (8.center) to (9.center);
		\draw [style=brace edge] (10.center) to (11.center);
		\draw [style=brace edge] (12.center) to (13.center);
		\draw (17) to (18.center);
		\draw (17) to (19.center);
		\draw (17) to (20.center);
		\draw (17) to (21);
		\draw [style=non edge] (2.center) to (17);
		\draw [style=non edge] (3.center) to (17);
	\end{pgfonlayer}
\end{tikzpicture}

%% file: figures/d_separable_h_0_h_1.tikz
\begin{tikzpicture}
	\begin{pgfonlayer}{nodelayer}
		\node [style=none] (0) at (-8.25, -0.25) {};
		\node [style=none] (1) at (-5.25, -0.25) {};
		\node [style=none] (2) at (-1.25, -0.25) {};
		\node [style=none] (3) at (1.75, -0.25) {};
		\node [style=none] (4) at (5.75, -0.25) {};
		\node [style=none] (5) at (8.75, -0.25) {};
		\node [style=none] (6) at (11.775, -0.25) {};
		\node [style=none] (7) at (-11.25, -0.25) {};
		\node [style=none] (18) at (-7.25, -0.25) {};
		\node [style=none] (19) at (-6, -0.25) {};
		\node [style=vertex] (21) at (7.75, -0.25) {};
		\node [style=none] (24) at (-3.5, -4.25) {};
		\node [style=none] (25) at (3.5, -4.25) {};
		\node [style=none] (26) at (0, -4.25) {$H_2$};
		\node [style=vertex] (27) at (-2.75, -4.25) {};
		\node [style=vertex] (28) at (2.5, -4.25) {};
		\node [style=none] (29) at (-2.75, -4.75) {$h_0$};
		\node [style=none] (30) at (2.5, -4.75) {$h_1$};
		\node [style=vertex] (31) at (6.75, -0.25) {};
		\node [style=vertex] (32) at (-8.25, -0.25) {};
		\node [style=vertex] (33) at (-7.25, -0.25) {};
		\node [style=none] (34) at (-8.25, 0.25) {$p_x$};
		\node [style=none] (35) at (-7.25, 0.25) {};
		\node [style=none] (36) at (6.75, 0.25) {};
		\node [style=none] (37) at (6.825, 0.275) {$p_y$};
	\end{pgfonlayer}
	\begin{pgfonlayer}{edgelayer}
		\draw (7.center) to (0.center);
		\draw (0.center) to (1.center);
		\draw (2.center) to (1.center);
		\draw (2.center) to (3.center);
		\draw (3.center) to (4.center);
		\draw (4.center) to (5.center);
		\draw (5.center) to (6.center);
		\draw [in=-90, out=-90, looseness=0.75] (24.center) to (25.center);
		\draw [in=90, out=90, looseness=0.75] (25.center) to (24.center);
		\draw (27) to (32);
		\draw (33) to (27);
		\draw (28) to (31);
		\draw (28) to (21);
	\end{pgfonlayer}
\end{tikzpicture}

%% file: figures/d_separable_gh.tikz
\begin{tikzpicture}
	\begin{pgfonlayer}{nodelayer}
		\node [style=vertex] (0) at (-5, 0) {};
		\node [style=vertex] (1) at (4, 0) {};
		\node [style=vertex] (4) at (-2, 0) {};
		\node [style=vertex] (5) at (2, 0) {};
		\node [style=vertex] (6) at (0, -0.5) {};
		\node [style=none] (7) at (0.575, -0.5) {$z_3$};
		\node [style=none] (8) at (-5, 0.725) {};
		\node [style=none] (9) at (-2.5, 0.725) {};
		\node [style=none] (10) at (-3.6, 1.05) {\tiny $P_L$};
		\node [style=none] (11) at (-2, -3) {};
		\node [style=none] (12) at (2, -3) {};
		\node [style=none] (13) at (-0.075, -3) {$B$};
		\node [style=none] (14) at (-0.75, -2.75) {};
		\node [style=none] (15) at (0.5, -2.75) {};
		\node [style=none] (16) at (0, -2.5) {};
		\node [style=none] (17) at (-5, -0.5) {\tiny $p_1$};
		\node [style=none] (18) at (4.05, 0.30) {\tiny $p_{k-1}$};
		\node [style=vertex] (19) at (5.1, 0) {};
		\node [style=none] (20) at (1, -2.75) {};
		\node [style=none] (21) at (1.25, -3.25) {};
		\node [style=none] (22) at (5.1, -0.575) {\tiny $p_k$};
		\node [style=none] (23) at (2.5, 0.7) {};
		\node [style=none] (24) at (5.1, 0.7) {};
		\node [style=none] (25) at (3.90, 1.05) {\tiny$P_R$};
		\node [style=none] (26) at (-1.75, 0.30) {$z_1$};
		\node [style=none] (27) at (1.5, 0.25) {$z_2$};
		\node [style=vertex] (28) at (-2.5, 0) {};
		\node [style=vertex] (29) at (2.5, 0) {};
	\end{pgfonlayer}
	\begin{pgfonlayer}{edgelayer}
		\draw (0) to (4);
		\draw (5) to (1);
		\draw [style=brace edge] (8.center) to (9.center);
		\draw [in=90, out=90, looseness=0.75] (11.center) to (12.center);
		\draw [in=-90, out=-90, looseness=0.75] (11.center) to (12.center);
		\draw (6) to (14.center);
		\draw (6) to (15.center);
		\draw (6) to (16.center);
		\draw (19) to (1);
		\draw [in=45, out=30, looseness=1.25] (20.center) to (21.center);
		\draw [in=-150, out=-150, looseness=1.50] (20.center) to (21.center);
		\draw (20.center) to (19);
		\draw (19) to (21.center);
		\draw [style=brace edge] (23.center) to (24.center);
		\draw [style=non edge] (4) to (11.center);
		\draw [style=non edge] (4) to (12.center);
		\draw [style=non edge] (5) to (11.center);
		\draw [style=non edge] (5) to (12.center);
	\end{pgfonlayer}
\end{tikzpicture}

%% file: figures/d_separable_notbothsides.tikz
\begin{tikzpicture}
	\begin{pgfonlayer}{nodelayer}
		\node [style=vertex] (0) at (-5, 0) {};
		\node [style=vertex] (1) at (5, 0) {};
		\node [style=vertex] (4) at (-2, 0) {};
		\node [style=vertex] (5) at (2, 0) {};
		\node [style=vertex] (6) at (0, -0.5) {};
		\node [style=none] (7) at (0.575, -0.5) {$z_3$};
		\node [style=none] (8) at (-5, 0.475) {};
		\node [style=none] (9) at (-2, 0.475) {};
		\node [style=none] (10) at (-3.5, 0.925) {\tiny $Q_1$};
		\node [style=none] (11) at (-2, -3) {};
		\node [style=none] (12) at (2, -3) {};
		\node [style=none] (13) at (-0.075, -3) {$B$};
		\node [style=none] (14) at (-0.75, -2.75) {};
		\node [style=none] (15) at (0.5, -2.75) {};
		\node [style=none] (16) at (0, -2.5) {};
		\node [style=none] (17) at (-5, -0.5) {$p_1$};
		\node [style=none] (18) at (5.9, -0.325) {$p_{k-1}$};
		\node [style=vertex] (19) at (4.5, -1) {};
		\node [style=none] (20) at (1, -2.75) {};
		\node [style=none] (21) at (1.25, -3.25) {};
		\node [style=none] (22) at (5, -1) {$p_k$};
		\node [style=none] (23) at (2, 0.45) {};
		\node [style=none] (24) at (3.75, 0.45) {};
		\node [style=none] (25) at (3, 0.9) {\tiny$Q_2'$};
		\node [style=none] (26) at (-1.5, 0.25) {$z_1$};
		\node [style=none] (27) at (1.5, 0.25) {$z_2$};
		\node [style=vertex] (29) at (3.75, 0) {};
		\node [style=none] (30) at (3.75, -0.5) {$f$};
	\end{pgfonlayer}
	\begin{pgfonlayer}{edgelayer}
		\draw (0) to (4);
		\draw (5) to (1);
		\draw [style=brace edge] (8.center) to (9.center);
		\draw [in=90, out=90, looseness=0.75] (11.center) to (12.center);
		\draw [in=-90, out=-90, looseness=0.75] (11.center) to (12.center);
		\draw (6) to (14.center);
		\draw [style=hadamard edge] (6) to (15.center);
		\draw (6) to (16.center);
		\draw (19) to (1);
		\draw [in=45, out=30, looseness=1.25] (20.center) to (21.center);
		\draw [in=-150, out=-150, looseness=1.50] (20.center) to (21.center);
		\draw (20.center) to (19);
		\draw (19) to (21.center);
		\draw [style=brace edge] (23.center) to (24.center);
		\draw [style=hadamard edge, in=210, out=-75, looseness=0.50] (15.center) to (19);
		\draw [style=hadamard edge] (19) to (1);
		\draw [style=hadamard edge] (1) to (29);
	\end{pgfonlayer}
\end{tikzpicture}

%% file: figures/boosted_separator.tikz
\begin{tikzpicture}
	\begin{pgfonlayer}{nodelayer}
		\node [style=none] (0) at (-2, 0) {};
		\node [style=none] (1) at (2, 0) {};
		\node [style=none] (2) at (0, 0.75) {\small $X$};
		\node [style=vertex] (3) at (-1.25, -0.25) {};
		\node [style=none] (4) at (-1.5, 0) {$x$};
		\node [style=none] (5) at (-3, -3) {};
		\node [style=none] (6) at (2.75, -3) {};
		\node [style=none] (7) at (0.75, -2.75) {\small $N(X)$};
		\node [style=none] (8) at (-2.25, -5) {};
		\node [style=none] (9) at (-2.25, -7.75) {};
		\node [style=none] (10) at (-2.25, -6) {\small $B$};
		\node [style=none] (11) at (-2.55, -3) {};
		\node [style=none] (12) at (-0.8, -3) {};
		\node [style=none] (13) at (-2.75, -5.25) {};
		\node [style=none] (14) at (-1.75, -5.25) {};
		\node [style=none] (15) at (-1.75, -3) {\small $I$};
	\end{pgfonlayer}
	\begin{pgfonlayer}{edgelayer}
		\draw [in=90, out=90] (0.center) to (1.center);
		\draw [in=-90, out=-90] (1.center) to (0.center);
		\draw [in=90, out=90, looseness=0.50] (5.center) to (6.center);
		\draw [in=-90, out=-90, looseness=0.50] (6.center) to (5.center);
		\draw (0.center) to (5.center);
		\draw (1.center) to (6.center);
		\draw [in=0, out=0] (8.center) to (9.center);
		\draw [in=180, out=180] (9.center) to (8.center);
		\draw [in=90, out=90, looseness=0.75] (11.center) to (12.center);
		\draw [in=-90, out=-90, looseness=0.75] (12.center) to (11.center);
		\draw (13.center) to (11.center);
		\draw (14.center) to (12.center);
		\draw (3) to (11.center);
		\draw (3) to (12.center);
	\end{pgfonlayer}
\end{tikzpicture}

%% file: figures/slll.tikz
\begin{tikzpicture}
	\begin{pgfonlayer}{nodelayer}
		\node [style=vertex ] (0) at (-1.5, 0) {};
		\node [style=vertex ] (2) at (0, -1.5) {};
		\node [style=vertex ] (3) at (1.5, 0) {};
		\node [style=vertex ] (4) at (0, 2) {};
		\node [style=none] (5) at (-0.5, 2) {\tiny $x$};
		\node [style=none] (6) at (-2, 0) {\tiny $a_1$};
		\node [style=none] (7) at (2, 0) {\tiny $a_2$};
		\node [style=none] (8) at (-0.3, -1.825) {\tiny $a_3$};
		\node [style=none] (9) at (-0.475, 0.5) {};
		\node [style=none] (10) at (0.7, -1.175) {};
		\node [style=none] (11) at (-0.65, -0.925) {};
	\end{pgfonlayer}
	\begin{pgfonlayer}{edgelayer}
		\draw [style=gray edge, in=-165, out=45, looseness=1.50] (0) to (3);
		\draw [style=gray edge, in=0, out=-150] (3) to (2);
		\draw [style=gray edge, in=165, out=0] (0) to (2);
		\draw (4) to (0);
		\draw (4) to (3);
		\draw (4) to (2);
		\draw [style=hadamard edge, in=165, out=45] (0) to (9.center);
		\draw [style=hadamard edge, in=75, out=-150, looseness=0.75] (3) to (10.center);
		\draw [style=hadamard edge, in=-75, out=165] (2) to (11.center);
	\end{pgfonlayer}
\end{tikzpicture}

%% file: figures/boosting.tikz
\begin{tikzpicture}
	\begin{pgfonlayer}{nodelayer}
		\node [style=none] (0) at (-3, 2) {};
		\node [style=none] (1) at (1, 2) {};
		\node [style=none] (2) at (-3, 1) {};
		\node [style=none] (3) at (1, 1) {};
		\node [style=none] (4) at (-5, -2) {};
		\node [style=none] (5) at (3, -2) {};
		\node [style=none] (6) at (-0.75, -2.25) {};
		\node [style=none] (7) at (2.25, -2.25) {};
		\node [style=none] (8) at (0.25, -5) {};
		\node [style=none] (9) at (1.75, -5) {};
		\node [style=none] (10) at (3.5, -5) {};
		\node [style=none] (11) at (5, -5) {};
		\node [style=none] (12) at (-1, 1.5) {};
		\node [style=none] (13) at (-1, 1.5) {$X$};
		\node [style=none] (14) at (0.75, -2.25) {$Z$};
		\node [style=none] (15) at (4.25, -5) {};
		\node [style=none] (16) at (4.25, -5) {\tiny big};
		\node [style=none] (17) at (1, -5) {\tiny big};
		\node [style=none] (20) at (-3, -5) {};
		\node [style=none] (21) at (-1.675, -5) {};
		\node [style=none] (22) at (-2.25, -5) {};
		\node [style=vertex set] (23) at (-2.35, -5) {\tiny small };
		\node [style=none] (24) at (-5, -5) {};
		\node [style=none] (25) at (-3.675, -5) {};
		\node [style=none] (26) at (-4.25, -5) {};
		\node [style=vertex set] (27) at (-4.35, -5) {\tiny small};
		\node [style=none] (29) at (-3.75, -2.25) {};
		\node [style=none] (30) at (-2.25, -2.5) {};
	\end{pgfonlayer}
	\begin{pgfonlayer}{edgelayer}
		\draw (0.center) to (1.center);
		\draw (1.center) to (3.center);
		\draw (3.center) to (2.center);
		\draw (2.center) to (0.center);
		\draw [in=90, out=90, looseness=0.50] (4.center) to (5.center);
		\draw [in=-90, out=-90, looseness=0.50] (4.center) to (5.center);
		\draw [in=90, out=90, looseness=0.75] (6.center) to (7.center);
		\draw [in=-90, out=-90, looseness=0.75] (6.center) to (7.center);
		\draw [in=90, out=90] (8.center) to (9.center);
		\draw [in=-90, out=-90] (8.center) to (9.center);
		\draw [in=90, out=90] (10.center) to (11.center);
		\draw [in=-90, out=-90] (10.center) to (11.center);
		\draw [bend right=15] (2.center) to (4.center);
		\draw [bend left=15, looseness=1.25] (3.center) to (5.center);
		\draw (6.center) to (8.center);
		\draw (7.center) to (9.center);
		\draw [bend right, looseness=0.75] (6.center) to (10.center);
		\draw [bend left, looseness=1.25] (7.center) to (11.center);
		\draw [in=90, out=90] (20.center) to (21.center);
		\draw [in=-90, out=-90] (20.center) to (21.center);
		\draw [in=90, out=90] (24.center) to (25.center);
		\draw [in=-90, out=-90] (24.center) to (25.center);
		\draw (25.center) to (29.center);
		\draw (30.center) to (21.center);
		\draw (24.center) to (29.center);
		\draw (20.center) to (30.center);
	\end{pgfonlayer}
\end{tikzpicture}

%% file: figures/path_caterpillar_line_graph_of_caterpillar.tikz
\begin{tikzpicture}
	\begin{pgfonlayer}{nodelayer}
		\node [style=Z phase dot] (0) at (-14.75, 0) {$x_1$};
		\node [style=Z phase dot] (1) at (-12.25, 0) {$x_2$};
		\node [style=Z phase dot] (2) at (-9.75, 0) {$x_3$};
		\node [style=Z phase dot] (3) at (-7, 0) {$x_4$};
		\node [style=Z phase dot] (5) at (-3.375, 0) {$x_h$};
		\node [style=none] (6) at (-4.95, 0) {};
		\node [style=none] (7) at (-6.3, 0) {};
		\node [style=none] (8) at (-5.6, 0) {$\cdots$};
		\node [style=Z phase dot] (9) at (1.5, 0) {};
		\node [style=Z phase dot] (12) at (7.5, 0) {};
		\node [style=Z phase dot] (13) at (10.625, 0) {};
		\node [style=none] (14) at (9.55, 0) {};
		\node [style=none] (15) at (8.2, 0) {};
		\node [style=none] (16) at (8.9, 0) {$\cdots$};
		\node [style=Z phase dot] (17) at (1.5, 3.75) {$x_1$};
		\node [style=Z phase dot] (18) at (3.5, 0) {};
		\node [style=Z phase dot] (19) at (3.5, 3.75) {$x_2$};
		\node [style=Z phase dot] (20) at (5.5, 0) {};
		\node [style=Z phase dot] (21) at (5.5, 3.75) {$x_3$};
		\node [style=none] (22) at (7.5, 0) {};
		\node [style=Z phase dot] (23) at (7.5, 3.75) {$x_4$};
		\node [style=none] (25) at (10.625, 0) {};
		\node [style=Z phase dot] (26) at (10.625, 3.75) {$x_h$};
		\node [style=Z phase dot] (28) at (32.475, 0) {};
		\node [style=Z phase dot] (29) at (35.35, 0) {};
		\node [style=none] (30) at (34.525, 0) {};
		\node [style=none] (31) at (33.175, 0) {};
		\node [style=none] (32) at (33.875, 0) {$\cdots$};
		\node [style=Z phase dot] (36) at (27.5, 0) {};
		\node [style=none] (38) at (32.475, 0) {};
		\node [style=Z phase dot] (40) at (36.4, 1.175) {};
		\node [style=Z phase dot] (41) at (36.35, 4) {$x_h$};
		\node [style=Z phase dot] (43) at (30.5, 0) {};
		\node [style=Z phase dot] (44) at (31.5, 1.175) {};
		\node [style=Z phase dot] (45) at (31.65, 4) {$x_4$};
		\node [style=Z phase dot] (46) at (32.475, 0) {};
		\node [style=Z phase dot] (49) at (25.55, 0) {};
		\node [style=Z phase dot] (50) at (26.525, 1.15) {};
		\node [style=Z phase dot] (51) at (26.675, 4) {$x_3$};
		\node [style=Z phase dot] (52) at (27.5, 0) {};
		\node [style=Z phase dot] (55) at (20.5, 0) {};
		\node [style=Z phase dot] (56) at (21.5, 1.175) {};
		\node [style=Z phase dot] (57) at (21.55, 4) {$x_2$};
		\node [style=Z phase dot] (58) at (22.525, 0) {};
		\node [style=Z phase dot] (62) at (16.475, 1.175) {};
		\node [style=Z phase dot] (63) at (16.55, 3.75) {$x_1$};
		\node [style=Z phase dot] (64) at (17.5, 0) {};
	\end{pgfonlayer}
	\begin{pgfonlayer}{edgelayer}
		\draw [in=-150, out=30, looseness=1.50] (0) to (1);
		\draw [in=150, out=-30, looseness=1.50] (1) to (2);
		\draw [in=165, out=-30, looseness=1.75] (2) to (3);
		\draw (3) to (7.center);
		\draw (6.center) to (5);
		\draw (12) to (15.center);
		\draw (14.center) to (13);
		\draw [in=120, out=-60, looseness=1.25] (17) to (9);
		\draw [in=120, out=-60, looseness=1.25] (19) to (18);
		\draw [in=-150, out=30, looseness=1.50] (9) to (18);
		\draw [in=120, out=-60, looseness=1.25] (21) to (20);
		\draw [in=165, out=-30, looseness=1.75] (20) to (12);
		\draw [in=150, out=-30, looseness=1.50] (18) to (20);
		\draw [in=120, out=-60, looseness=1.25] (23) to (22.center);
		\draw [in=120, out=-60, looseness=1.25] (26) to (25.center);
		\draw (28) to (31.center);
		\draw (30.center) to (29);
		\draw [in=120, out=-60, looseness=1.50] (41) to (40);
		\draw (29) to (40);
		\draw [in=120, out=-60, looseness=1.25] (45) to (44);
		\draw (43) to (46);
		\draw (44) to (46);
		\draw (43) to (44);
		\draw [in=120, out=-60, looseness=1.25] (51) to (50);
		\draw (49) to (52);
		\draw (50) to (52);
		\draw (49) to (50);
		\draw [in=60, out=-105, looseness=1.50] (57) to (56);
		\draw (55) to (58);
		\draw (56) to (58);
		\draw (55) to (56);
		\draw [in=60, out=-120, looseness=1.50] (63) to (62);
		\draw (62) to (64);
		\draw [in=-150, out=30, looseness=1.50] (52) to (43);
		\draw [in=-150, out=30, looseness=1.75] (58) to (49);
		\draw [in=-150, out=30, looseness=1.75] (64) to (55);
	\end{pgfonlayer}
\end{tikzpicture}

%% file: figures/claw_case.tikz
\begin{tikzpicture}
	\begin{pgfonlayer}{nodelayer}
		\node [style=vertex] (13) at (-7.375, 1.425) {};
		\node [style=vertex] (14) at (-8.75, 1.375) {};
		\node [style=vertex] (15) at (-6.4, 1.45) {};
		\node [style=vertex] (16) at (-5.125, 1.45) {};
		\node [style=vertex] (20) at (-6.9, -2) {$\ \ $};
		\node [style=none] (21) at (-8.025, 0.65) {};
		\node [style=none] (22) at (-7.125, 0.925) {};
		\node [style=none] (23) at (-6.1, 0.65) {};
		\node [style=none] (24) at (-8.75, 1.95) {$x_1$};
		\node [style=none] (25) at (-7.375, 1.95) {$x_2$};
		\node [style=none] (26) at (-6.35, 1.95) {$x_3$};
		\node [style=none] (27) at (-5.1, 1.95) {$x_4$};
		\node [style=none] (28) at (-6.45, -2) {$z$};
	\end{pgfonlayer}
	\begin{pgfonlayer}{edgelayer}
		\draw [in=-30, out=150] (20) to (14);
		\draw [in=-120, out=45, looseness=1.25] (20) to (16);
		\draw [in=-60, out=120, looseness=1.25] (20) to (13);
		\draw [in=-45, out=90, looseness=0.75] (20) to (15);
		\draw [style=hadamard edge, in=150, out=-75] (21.center) to (20);
		\draw [style=hadamard edge, in=120, out=-75, looseness=1.25] (22.center) to (20);
		\draw [style=hadamard edge, in=90, out=-105, looseness=0.75] (23.center) to (20);
	\end{pgfonlayer}
\end{tikzpicture}